\documentclass[10pt,a4paper]{article}
\usepackage{biblatex}
\usepackage[margin=1.2in,width=150mm,top=35mm,bottom=25mm,bindingoffset=6mm]{geometry}
\usepackage[utf8]{inputenc}
\usepackage[english]{babel}
\usepackage{graphicx}
\usepackage{amsmath}
\usepackage{amsthm}
\usepackage{amsmath,amssymb} 
\usepackage{booktabs} 
\usepackage{listings} 
\usepackage{xcolor} 
\usepackage{courier} 
\usepackage{enumitem} 
\newtheorem{theorem}{Theorem}[section]

\newtheorem{lemma}[theorem]{Lemma}
\newtheorem{proposition}[theorem]{Proposition}
\newtheorem{remark}[theorem]{Remark}

\bibliography{ref}
\title{On Partition Functions for Time-Inhomogeneous Branching Random Walks}
\author{Qianrun Wu
  \thanks{Email: qianrun.wu@dauphine.eu}
  }

\begin{document}
\maketitle
\begin{abstract}
    We establish the phase transition and universality for the partition function of time-inhomogeneous branching random walks (BRWs) with decreasing variance increment, a model related to two-dimensional directed polymers. By modifying Berestycki’s GMC framework and adapting it to discrete BRWs, we prove that the critical value of time-inhomogeneous BRWs coincides with that of time-homogeneous ones, and the partition functions converge in $L^1$ in the subcritical regime. We also extend the universality at the critical parameter, showing the same decay rate of partition functions. Our approach reveals the potential for such framework beyond GMC, which provides a new technical path for other models, such as BRWs and Continuous Random Energy Model (CREM). Finally, we raise some open problems regarding the further implications of our results in other related models discussed in the present paper.
\end{abstract}

\section{Introduction} 
Branching random walks (BRWs) are a natural model in probability theory, motivated by their biological applications, that describe the evolution of a population of particles with spatial motion. Much is known about the phase transition of their partition function (the so-called Biggins martingale); see \cite{Shi2015} for precise statements and proofs. More recently, their time-inhomogeneous counterpart, first introduced in \cite{Ming-Zeitouni}, initiates the study of the asymptotic behaviors of BRWs with increments that exhibit a macroscopic change in variance, revealing differences in extreme value statistics. The series of works on this model raises a more fundamental question: How do environmental fluctuations impact the asymptotic behavior of BRWs? In the present paper, we attempt to answer this question from a qualitative perspective by studying its partition function.

The partition function of time-inhomogeneous BRWs with monotonically decreasing variance—our primary focus in the present paper—is motivated by and closely related to two-dimensional directed polymers. On the other hand, the model of Gaussian Multiplicative Chaos (GMC), originaed in \cite{Kahane} has long been viewed as a continuous analogue of BRWs, and numerous works connect these two models: for instance, many results from one often inspire analogous results for the other, and many approaches used in BRWs can be applied to those in GMC-type objects (see, for example \cite{2020} and \cite{gaussianfield} for some surveys of the connections between BRWs and two-dimensional Gaussian Free Field, a special case of GMC). In the present paper, we aim to proceed in the opposite direction: we use the framework in recent developments in GMC to further the study of BRWs.

We reveal the complexity of the partition function in time-inhomogeneous environments, a more intricate topic since the classical martingale structure fails and the correlation function lacks sufficient regularity—issues largely avoided in existing literature on BRWs and directed polymers. Fortunately, recent advances in GMC extend the theory to non-martingale constructions of GMC measures. In particular, Berestycki \cite{berestycki2017elementaryapproachgaussianmultiplicative} offers a novel perspective by analyzing the ``thick points'' (roughly, points with a prescribed growth rate in the approximation) that dominate the total mass.

Indeed, this framework has been applied to related objects in the GMC literature (with significant technical differences across distinct settings), including convergence of critical GMC under convolution approximation in \cite{Powell2018} and derivatives of subcritical GMC in \cite{jego2026derivativesgaussianmultiplicativechaos}. Our results demonstrate the broad applicability of the framework from \cite{berestycki2017elementaryapproachgaussianmultiplicative}, which extends to models beyond GMC.

Motivated by this framework, we give an alternative proof of the classical phase transition result for BRWs, carry out an analogous construction in the BRW setting, and establish universal behavior of the partition functions: in the subcritical regime, the time-homogeneous and time-inhomogeneous partition functions converge to the same non-degenerate random variable. These results imply that, when the system behaves regularly in early-time evolution, the fundamental asymptotic property of the partition function is determined solely by the underlying branching process and is robust to time inhomogeneity in weight increments. The core of our approach is that fluctuations of the partition function are dominated by the early generations of the BRW.

Indeed, such a finding has been verified in a closely related model, the continuous random energy model (CREM), first introduced in \cite{BOVIER2004481}, where a phase transition of free energy occurs with the critical parameter determined by the initial correlation. In the present paper, we further investigate this phase transition regarding the non-degeneracy of the limit, which serves as a direct consequence of the framework discussed herein (and will be elaborated on in Section \ref{CREM}), where we provide a characterization of the non-degeneracy of the limit in the weak-disorder regime under certain conditions on the correlation structure.

Looking ahead, results in this paper suggest that within the GMC framework, the correlation function of time-inhomogeneous BRWs is not log-correlated, implying that the GMC framework can be extended to more general settings. Additionally, our results on the time-inhomogeneity of increments align with a fundamental question in \cite{Universality} regarding how environmental fluctuations affect phase transitions (or other qualitative behavior) in disordered systems; see \cite{caravenna2026disorderedsystemscritical2d} for a recent survey on 2D stochastic heat flow (a continuous analogue of two-dimensional directed polymers) related to this question.

Heuristically, disordered systems are divided into two categories: disorder relevant and disorder irrelevant. Disorder relevance means arbitrarily small disorder (or random environment) changes the model’s large-scale qualitative behavior (e.g., critical exponent changes) and leads to a different scaling limit; disorder irrelevance means such changes occur only for sufficiently strong disorder. At the boundary (typically critical dimension), disorder perturbation has a more subtle effect, being either marginally relevant or irrelevant. This paper also shows that the partition function has the same asymptotic behavior as its time-homogeneous counterpart at the critical parameter.

\subsection{Outline:} In Section \ref{modelBRWs}, we introduce the model of BRWs (in time-homogeneous and time-inhomogeneous environments), state our main results, as well as the potential connections to related models, such as GMC/two-dimensional directed polymers and CREM. In Section \ref{Sub-Critical}, we will prove Theorem \ref{main2}. Then we will prove the upper bound of the critical parameter $\bar{\beta}_c$ in Section \ref{upper}, and together these results show that $\bar{\beta}_c=\beta_c$. Then we will discuss the application of our main result, from which we can see that Theorem \ref{Main1} serves as a direct consequence of the proof, and we can also extend the idea to the CREM model, as discussed in Proposition \ref{CREM}. In Section \ref{Critical}, we extend the universality to the critical parameter, showing the same polynomial decay rate of partition functions for time-inhomogeneous BRWs (Theorem \ref{main3}). Finally, we will discuss our results in context and raise some other potential directions for further study in Section \ref{Future}.

\subsection{Acknowledgment:} The author wishes to thank Clément Cosco for proposing this interesting model, as well as discussions and comments on the early stages of this work. The author is also very grateful to Jiaqi Liu and Bastein Mallein for discussions about these results in the broader context of modern research.

\section{Descriptions of Models and Main Results}\label{modelBRWs}

\subsection{The Model of Branching Random Walks} In the present paper, we consider a branching random walk on the real line $\mathbb{R}$. Initially, a particle sits at the origin. Its children form the first generation; their displacements from the origin correspond to a point process on the line (which could also be described by a Galton-Watson process, abbreviated as G-W process or G-W trees). These children have children of their own (who form the second generation) and behave—relative to their respective positions—like independent copies of the initial particle, and so on. We will focus on time-homogeneous and time-inhomogeneous Gaussian Branching Random Walks (BRWs) on a supercritical Galton-Watson tree $\mathbb{T}$ rooted at $o\in\mathbb{T}$ with mean offspring distribution equal to $\mathbb{E}[D_1]=d>1$ (we denote $D_n$ to be the particles at generation $n$), and the detailed definition of the model will be introduced below. In the rest of the paper, we will use the survival probability of the G-W tree:\begin{equation*}
     \mathbf{P}^*=\mathbf{P}(\cdot |\text{non-extinction)}
 \end{equation*}
 
 We assign to each vertex $\mathbf{x}\in\mathbb{T}$ an independent random variable $w(\mathbf{x})$, such that $\{w(\mathbf{x})\}_{\mathbf{x}\in\mathbb{T}}$ are i.i.d.\ $\mathcal{N}(0,1)$ (For $\sigma>0$, we denote $\mathcal{N}(\mu,\sigma^2)$ to be the Gaussian distribution with mean $\mu$ and variance $\sigma^2$.). For each $\mathbf{x}\in\mathbb{T}$, we write $|\mathbf{x}|=n$ if $v$ belongs to the $n$-th generation of the Galton–Watson tree, and we let $(x_1,\dots,x_n)$ be the unique geodesic path from $o$ to $v$.  Since the geodesic is unique we can refer to each branch at generation $n$ by the last vertex.

Let $W_n^\beta$ denote the (normalized) partition function for time-homogeneous BRWs at generation $n$, where $\beta\geq0$ is a fixed parameter (which is so-called the inverse temperature in statistical mechanics), defined by
\[
W_n^\beta=\frac{1}{d^n}\sum_{|\mathbf{x}|=n} e^{\beta H_n(\mathbf{x})-\frac{1}{2}\beta^2 n},
\]
where $H_n(\mathbf{x})=\sum_{i=1}^n w(x_i)$ is the Hamiltonian associated with the branch $\mathbf{x}=(x_1,\dots,x_n)$. It is well known that the family of random variables $(W_n^\beta)_{n\geq 0}$ is a unit-mean martingale, and thus converges almost surely to a nonnegative limit $W_\infty^\beta$. Applying Kolmogorov’s 0-1 law gives the usual dichotomy that exactly one of the events:\begin{equation*}
     \{W_\infty^\beta >0\} \text{ (weak disorder) and } \{W_\infty^\beta=0\} \text{ (strong disorder)}
\end{equation*}
has probability 1 (a similar argument also applies to the partition function of time-inhomogeneous branching random walks, a model will be introduced later). 

The following classical result of Biggins (see e.g.\ \cite{Shi2015} for a more general statement and proof) characterizes the non-degeneracy of $W_\infty^\beta$:

\begin{theorem}[Biggins Martingale Convergence Theorem]\label{Main1}
\textsc{(Phase Transition for Time-homogeneous BRWs)}
For a BRW on a supercritical Galton–Watson tree with mean offspring number $\mathbb{E}[D_1] = d > 1$, there exists a phase transition for the non-degeneracy of $W_\infty^\beta$, i.e.,
\[
\beta_c := \sup\Bigl\{\beta \geq 0 \mid \mathbb{P}\{W_\infty^\beta = 0\} > 0\Bigr\}
\]
is well-defined and equals $\sqrt{2\log d}$.
\end{theorem}

The classical proof of the previous result (see, for example, Chapter 5 of \cite{Shi2015}) relies on a so-called ``spine method'' (or ``spine decomposition''), an extensively used technique in the branching process literature, where one uses a change-of-measure argument and tilts the measure towards a particle chosen uniformly at random. Leveraging the martingale property, it turns out that under the tilted measure, the chosen particle dominates the total mass. 

\subsection{Gaussian multiplicative Chaos}

Inspired by the long-standing relationship between BRWs and Gaussian Multiplicative Chaos (GMC), we provide a new proof of this result in the present paper by applying the analogous framework and approach from Berestycki’s paper \cite{berestycki2017elementaryapproachgaussianmultiplicative} regarding the convergence of sub-critical GMC—this proof does not rely on the martingale property. The GMC framework in \cite{berestycki2017elementaryapproachgaussianmultiplicative} is chosen due to its natural connection with BRWs (BRWs serve as the discrete analogue of continuous GMC) and its martingale-free property, which perfectly aligns with time-inhomogeneous BRWs (which lack a martingale structure) and will be introduced later in this section.

The relationship between BRWs and GMC can be briefly described as follows (roughly speaking, we embed a binary BRWs into the unit interval, which yields so-called ``multiplicative cascades'', a preliminary version of GMC): First, let $\tilde{I}$ denote the infinite union of sub-intervals of $[0,1]$ defined by
\begin{equation*}
\tilde{I} = \bigcup_{m \geq 1} \tilde{I}_m, \text{ where } \tilde{I}_m:= \bigcup_{i=0}^{2^m - 1} \left[ \frac{i}{2^m}, \frac{i+1}{2^m} \right],
\end{equation*}
Then for any $I \in \tilde{I}_k$ (with $k \in \mathbb{N} \cup \{0\}$), set $W_I = e^{\beta W_I' - \frac{\beta^2}{2}}$ for some $\beta>0$, where $W_I' \sim \mathcal{N}(0,1)$ and the $W_I'$ are independent across distinct $I \in \tilde{I}$. For any point $x \in [0,1]$, define $w_n(x) = \prod_{I \in \tilde{I}_n, \, x \in I} W_I$. We then introduce the random measure $\mu_n:=\mu(n, \beta)$ satisfying: $\frac{\text{d}\mu_n}{\text{d}\lambda}  = w_n$, where $\lambda$ is the Lebesgue measure on $[0,1]$. For each $n \geq 1$, $\mu_n$ defines a random measure on $[0,1]$, and by independence of $W_I$ we can see that $\mathbb{E}[\mu_n([0,1])] =\prod_{I\in \tilde{I}_n}\mathbb{E}[W_I]= 1$ for all $n$. This construction can be extended to GW trees via analogous ideas (according to each sampling of G-W tree), which will be used in the proof of Theorem \ref{main3}.

The modern construction of GMC, originates from Kahane \cite{Kahane}, can be roughly described as a random measure in a given domain $D\subset\mathbb{R}^d$, associated to a centered Gaussian field $h$ (a generalized function, not defined point-wise) and a constant parameter $\gamma>0$. Subsequently, such measures were generalized to a continuous setting on $\mathbb{R}^d$, yielding a Borel measure on a given domain $D \subset \mathbb{R}^d$ of the form
\begin{equation}\label{GMC measure}
\mu(dx) = e^{\gamma h(x) - \frac{1}{2}\gamma^2 \mathbb{E}[h(x)^2]} \, \sigma(dx).
\end{equation}
In the expression above:\begin{itemize}
    \item $h$ is a centered Gaussian generalized random field with covariance kernel $K(x,y) = \log |x - y|^{-1} + g(x,y)$, where $g(x,y)$ is continuous on $\bar{D} \times \bar{D}$;
    \item $\gamma > 0$ is a fixed constant;
    \item $\sigma$ is a Radon measure on $\mathbb{R}^d$, serving as the reference measure.
\end{itemize}

In Berestycki’s paper \cite{berestycki2017elementaryapproachgaussianmultiplicative}, the author provides an elementary proof for the convergence of approximating measures $\mu_\epsilon$ constructed via the convolution of $h(x)$ with certain regularizing kernels $\theta_\epsilon$ to a non-trivial limiting measure $\mu$ in the sub-critical regime (specifically, when the parameter satisfies $\gamma<\sqrt{2d}$) and further demonstrates that this limiting measure $\mu$ is independent of the choice of regularizing kernels. The GMC framework with thick point analysis, i.e. the points $x\in D$ satisfying:\begin{equation*}
    \lim_{\epsilon\to 0} \frac{X_\epsilon(x)}{|\log 1/\epsilon|}=\gamma 
\end{equation*}
which was shown in \cite{berestycki2017elementaryapproachgaussianmultiplicative} Lemma 3.4 that they dominates the total GMC measure, and measure tilting argument used in that paper will serve a more suitable technical tool for phase transition analysis for BRWs in time-inhomogeneous environments, which will be the main focus in the present paper. The later model was initially discussed in Fang and Zeitouni's paper \cite{Ming-Zeitouni} and drew attention from researchers in multiple related areas. In particular, recent study in \cite{cosco2025centrallimittheoremlogpartition} suggests that there is an analogy between two-dimensional directed polymers and BRWs with monotone decreasing variance, which serves as one of the primary motivation for the present paper. 

\subsection{Two-dimensional Directed Polymers and Time-inhomogeneous BRWs}

Directed polymers in random environment describe the behavior of a long directed chain of monomers in presence of random impurities, see \cite{FC16} for a comprehensive introduction to this area. In the setting related to our discussions, the trajectory of the polymer is given by a nearest-neighbor path $(S_n)_{n\in\mathbb{N}}$ on two-dimensional lattice, while the impurities (also called the environment) are given by i.i.d.\ centered random variables $\omega(n,x)$, $n\in \mathbb{N}, x\in \mathbb{Z}^d$ of variance one satisfying the free energy $\lambda(\beta)=\log \mathbb{E}[e^{\beta \omega(n,x)}]\in(0,\infty)$ for all $\beta >0$, and define $W_N(x,\beta) = \mathbb{E}_x\left[e^{\sum_{n=1}^N \left(\beta \omega(n,S_n)-\lambda(\beta)\right)}\right]$ to be the partition function for paths start at $x$. In the seminal paper \cite{Universality}, Caravenna, Sun and Zygouras established the convergence of the spatial behavior of two-dimensional (re-scaled) partition functions to a Gaussian random variable (see \cite{Universality}, Theorem 2.5 for a precise statement).  

In \cite{cosco2025centrallimittheoremlogpartition}, the authors showed that when the environment are i.i.d. standard Gaussian (i.e. having law $\mathcal{N}(0,1)$), the normalized partition function can be heuristically viewed as a branching random walk with decreasing variance (where each branch corresponds to a point $x\in\mathbb{Z}^d$), satisfy that $X_k(x)$ and $X_k(y)$ are correlated for small $|x-y|$ ($x,y\in\mathbb{Z}^2$), and decorrelate as $|x-y|$ grows large. Consequently, the asymptotic behavior of the partition function should be governed by a branching random walk model with decreasing variance discussed in the present paper. 

In time-inhomogeneous environment, the weight increment of any branch $\mathbf{|x|}=n$ exhibits a macroscopic variation: let $f:[0,1]\to [0,1]$ be continuous at 0, monotone decreasing function satisfying $f(0)=1$ and $f(1)=0$. For each generation in G-W tree $h\in \{0,...,n\}$, the weight increment at step $h$ follows the distribution of $\sqrt{f(\frac{h}{n})}w(x_h)$. In other words, the Hamiltonian of a branch $\mathbf{x}=\{x_1,...,x_n\}$, denoted by $\bar{H}_n(\mathbf{x})$, is given by: $\bar{H}_n(x)=\sum_{i=1}^{n} \sqrt{f(\frac{i}{n})}w(x_i)$, so that $\bar{H}_n(\mathbf{x})\sim\mathcal{N}(0,\sum_{i=1}^{n} f(\frac{i}{n}))$. Analogous to the normalized partition function for the BRW with homogeneous weight increments, we define the normalized partition function for the time-inhomogeneous BRW of length $n$, denote $\bar{W}_n^\beta$, as follows:\begin{equation*}
     \bar{W}_n^\beta=\frac{1}{d^n}\sum_{\mathbf{|x|}=n} e^{\beta \bar{H}_n(x)-\frac{\beta^2}{2}\sum_{i=1}^{n} f(\frac{i}{n})w(x_i)}
 \end{equation*}
However, unlike $(W_n^\beta)_{n\geq 0}$ in time-homogeneous environment, the partition function for time-inhomogeneous environment is not a martingale (due to the non-constancy of $f$), so the convergence of $(W_n^\beta)_{n\geq 0}$ is not automatic. 

In the present paper, we will establish the following stronger result: when $\beta<\beta_c$, the difference $\bar{W}_n^\beta-W_n^\beta\to 0$ in the $L^1$-norm, which in turn implies convergence to 0 in probability. Notably, the limit of $\bar{W}_n^\beta$ ,denoted by $\bar{W}_\infty^\beta$, follows the same law as $W_\infty^\beta$-the a.s. limit of partition function in time-homogeneous weights. As a direct consequence, we can see that $\bar{W}_\infty^\beta$ is almost surely non-negative when $\beta<\sqrt{2\log d}$:

\begin{theorem} \label{main2}
\textsc{(Phase Transition and Universality for Time-inhomogeneous BRW)}
Let $f:[0,1]\to[0,1]$ be a continuous, non-increasing function with $f(0)=1$. Define the normalized partition function for time-inhomogeneous weights by
\[
\bar{W}_n^\beta=\frac{1}{d^n}\sum_{|\mathbf{x}|=n} e^{\beta \bar{H}_n(\mathbf{x})-\frac{\beta^2}{2}\sum_{i=1}^n f\!\left(\frac{i}{n}\right)},
\]
where
\[
\bar{H}_n(\mathbf{x})=\sum_{i=1}^n \sqrt{f\!\left(\frac{i}{n}\right)}w(x_i)
\]
is the Hamiltonian for a branch $\mathbf{x}$ with $|\mathbf{x}|=n$. Then $\bar{W}_n^\beta$ exhibits a phase transition for its limit in probability (whose existence is established separately), with critical value
\[
\bar{\beta}_c := \sup\Bigl\{\beta \geq 0 \mid \mathbb{P}\{\bar{W}_\infty^\beta = 0\} > 0\Bigr\},
\]
which is well-defined and equals $\sqrt{2\log d}$ (coincides with that of the homogeneous model $W_n^\beta$, i.e., $\beta_c = \bar{\beta}_c$). Moreover, the following universality in $L^1$ limit holds:
\[
\mathbb{E}\bigl[|W_n^\beta - \bar{W}_n^\beta|\bigr] \to 0
\quad\text{as }n\to\infty,\text{ for all }\beta<\beta_c=\sqrt{2\log d}.
\]
\end{theorem}

Finally, the exact asymptotic behavior at criticality was established in the influential work of Hu and Shi \cite{HS09},who proved that at the critical parameter $\beta_c$, $W_n^\beta$ possess the asymptotic behavior:\begin{equation*}
    W_n^\beta=n^{-\frac{1}{2}+o(1)} \text{ a.s. } 
\end{equation*}
 Our next theorem shows that the universality principle extends to the critical parameter. This result was also motivated by the paper \cite{alberts2012nearcriticalscalingwindowdirected},which studies the decay rate of partition function near criticality:
 
 \begin{theorem} \label{main3}[Asymptotics of partition function at criticality]
    At the critical parameter $\beta_c$, we have $\bar{W}_n^{\beta_c}=n^{-\frac{1}{2}+o(1)}$ \text{almost surely}. 
\end{theorem}

The polynomial decay rate for time-inhomogeneous BRWs at criticality originates from the universal domination of early generations, which is consistent with time-homogeneous BRWs and reflects the insensitivity of phase transition to variance inhomogeneity.

\subsection{Continuous Random Energy Model}

Indeed, while the property of $f$ in the statement of Theorem \ref{main2} was our primary motivation from two-dimensional directed polymers, from the proof in Section \ref{Sub-Critical} we can see that the restriction on $f$ in Theorem \ref{main2} could be relaxed (i.e., Theorem \ref{main2} holds for a generalized class of $f$): the proof will hold as long as $f:[0,1]\to[0,1]$ with $f(0)=1$, is continuous at $x=0$.

Motivated by the mild conditions on $f$, our result could also be extended to another closely related model: the continuous random energy model (CREM), which was introduced by Bovier and Kurkova in the paper \cite{BOVIER2004481} in 2004, which serves as another toy model of disordered systems exhibiting an infinitely hierarchical correlation structure; see also \cite{Ho23} and \cite{lee2025samplingcontinuousrandomenergy} for more recent studies of this model. The model is defined as a generalized version of binary time-homogeneous BRWs: denote by $\mathbb{T}_n$ the binary tree with depth $n$, and given $\mathbf{x},\mathbf{y}\in\mathbb{T}_n$, $C(\mathbf{x},\mathbf{y})$ is the depth of their most recent common ancestor (see the precise definition in the proof of Theorem \ref{l2inhomo}). The CREM is a centered Gaussian process indexed by the binary tree $\mathbb{T}_n$ of depth $N$ (or equivalently, on the hypercube $\{0,1\}^N$ in other literature, and it serves as a toy model of spin glass on $\{0,1\}^N$ in this case) with covariance function:
\begin{align*}
\text{Cov}\left(\mathbf{x},\mathbf{y}\right)= n\cdot A\left(\frac{C(\mathbf{x},\mathbf{y})}{n}\right), \quad \forall |\mathbf{x}|,|\mathbf{y}|=n.
\end{align*}
where $A:[0,1]\to [0,1]$ is a non-decreasing function satisfying $A(0)=0$ and $A(1)=1$, and the right derivative of $A$ at $0$, denoted by $\hat{A}'(0)$, is finite (i.e., $\hat{A}'(0)<\infty$).

To study the CREM, one of the key quantities is the normalized partition function defined as $Z_n^\beta$ (the quantity $\beta_c$ is sometimes referred to as the static critical inverse temperature of the CREM):
\begin{align}
\label{CREM par func}
Z_n^{\beta} := \frac{\sum_{|\mathbf{x}|=n} e^{\beta H_n(\mathbf{x})}}{\mathbb{E}[\sum_{|\mathbf{x}|=n} e^{\beta H_n(\mathbf{x})}]}
\end{align}

From the expression of the covariance function, we can see that it possesses a similar construction to time-inhomogeneous BRWs on binary trees. Indeed, one may view $A(x)$ as $\int_0^x f(t) dt$ in the time-inhomogeneous BRWs for any $x\in [0,1]$, which serves as the limit of the discrete sum of $f(\frac{i}{n})$ (whereas the case $A(x)=x$ corresponds to time-homogeneous BRWs). In the same paper, Bovier and Kurkova showed that there exists a critical parameter $\beta_c=\frac{\sqrt{2\log 2}}{\sqrt{\hat{A}'(0)}}$ such that the following phase transition occurs: heuristically speaking, for all $\beta< \beta_c$, the main contribution to the partition function $Z_n^\beta$ comes from an exponential number of the particles, and for all $\beta>\beta_c$, the maximum starts to contribute significantly to the partition function. 

However, to our knowledge, it has not yet been proven in the literature that $\beta_c$ serves as the critical parameter for the non-degeneracy of the limit of $Z_n^\beta$. Under the framework of the present paper, we can see that although the exact approximation terms of $Z_n^\beta$ are different from those of $\bar{W}_n^\beta$, the similarity in Gaussian increments and the covariance structure allows us to prove the convergence of the partition function under certain conditions on $A$:

\begin{proposition} [Convergence of CREM Partition function]\label{CREMPARFUN}
    Suppose for any $x\in [0,1]$, $A'(x)$ exists almost everywhere, satisfies $\sup_{x\in [0,1]}A'(x)\leq A'(0)$, and is continuous at $0$ (i.e., $A\in  C^0[0,1]\cap C^1(0)$). Then when $\beta<\beta_c=\frac{\sqrt{2\log 2}}{\sqrt{A'(0)}}$, we have $Z_n^\beta$ converges to $Z_\infty^\beta$ in $L^1$ and $Z_\infty^\beta$ has the same law as $W_\infty^\beta$.
\end{proposition}

Many other interesting results have been established for the CREM and its partition function, and we believe that it would be interesting to compare those results with the BRWs literature (including the results in the present paper).

\section{Proof of Main Results} 
\subsection{Proof of Theorem \ref{main2}}\label{Sub-Critical}

In this section, we establish the existence of critical parameters governing the phase transitions of the limit $\bar{W}_\infty^\beta$, denoted respectively $\bar{\beta_c}$, and further derive its exact value. 

\subsubsection{Lower Bound for Time-inhomogeneous Critical Parameter}

The lower bound follows from a similar idea in the construction of sub-critical regime of GMC measure in \cite{berestycki2017elementaryapproachgaussianmultiplicative}. Before we dive into the proof, we would like to illustrate a standard summation argument (the $L^2$-computation) used in our analysis, which usually forms the main technicality in various research in branching processes, and will be used repeatedly (in various forms) throughout the paper. For simplicity, we would first focus on our discussion in the time-homogeneous weights, and it would be straightforward to extend the computation to the time-inhomogeneous weights, by independence of increments. The key idea is to fix a branch $\mathbf{|x|}=n$ and sum over all possible depths for $|\mathbf{y}|=n$ at which the pair of branches $\mathbf{x}$ and $\mathbf{y}$ diverge. 

Within the BRWs framework, the covariance $\text{Cov}\left(H_n(\mathbf{x}), H_n(\mathbf{y})\right)$ for two branches  branches $\mathbf{x} = (x_1, \dots, x_n)$ and $\mathbf{y} = (y_1, \dots, y_n)$ depends exclusively on the depth at which they first diverge, denoted by $C(\mathbf{x},\mathbf{y})$:
\begin{equation*}
C(\mathbf{x},\mathbf{y}) 
:= \sup \left\{ k \geq 0 \mid x_k = y_k \right\}
\end{equation*}
with the convention that $C(\mathbf{x},\mathbf{y}) = 0$ if the set is empty. For such pair of branches satisfying $C(\mathbf{x},\mathbf{y})=h$, the covariance of their Hamiltonian is $\text{Cov}\left(H_n(\mathbf{x}), H_n(\mathbf{y})\right) = 2h$.

To illustrate the $L^2$ norm, we fix a branch $\mathbf{|x|}=n$ (noting there are expected $d^n$ such branches) and sum over all possible branching depths $h \in \{0, \dots, n\}$ relative to $\mathbf{x}$. For each fixed $\mathbf{x}$ and depth $h$, the number of branches $\mathbf{y}$ satisfying $C(\mathbf{x},\mathbf{y}) = h$ has expected value $d^{n-h}$, corresponding to the number of branches in a sub-Galton-Watson tree with   $(n-h)$ generations.

This structure allows us to rewrite the $L^2$-norm of $W_n^\beta$ as:
\begin{equation*}
    \mathbb{E}\left[(W_n^\beta)^2\right]=d^{-2n}\mathbb{E}\left[
    \sum_{\mathbf{|x|=n}}\sum_{h=0}^n \sum_{C(\mathbf{x,\mathbf{y}})=h}[e^{-n\beta^2+\beta (H_n(\mathbf{x})+H_n(\mathbf{y}))}]
    \right]
\end{equation*}
by independence of the law of GW tree and the weight functions, we have:
\begin{equation*}
    \mathbb{E}\left[(W_n^\beta)^2\right]=\sum_{h=0}^n d^{-h}\mathbb{E}[e^{-n\beta^2+\beta (H(\mathbf{x})+H(\mathbf{y}))}]
\end{equation*}
\begin{equation*}
   =\sum_{h=0}^n d^{-h}e^{-n\beta^2}e^{\beta^2 (n+h)}= \sum_{h=0}^n e^{h(\beta^2-\log d)}
\end{equation*}  
and finally we can see that the series converges if and only if $\beta^2-\log d<0$, i.e., $\beta<\sqrt{\log d}$. 

Then we will establish the existence of the $L^2$ critical parameter for the time-inhomogeneous weights, denoted $\bar{\beta}_2$:

\begin{theorem}\label{L^2}[Time-inhomogeneous $L^2$ Critical Parameter $\bar{\beta}_2$]

The $L^2$-critical parameter in time-inhomogeneous environment, defined by:
\begin{equation*}
\bar{\beta}_2 := \sup_{\beta \geq 0} \left\{ \sup_{n \geq 0} \mathbb{E}\left[ \left( W_n^\beta \right)^2 \right] < \infty \right\}
\end{equation*}
exists and equals to $\bar{\beta}_2 = \sqrt{\log d}$ (in particular, the $L^2$ critical parameters in both cases coincide).
\end{theorem}

\begin{proof}
Following the previous computations, we may express the definition of $\bar{\beta}_2$ in terms of the $L^2$-norm of $\bar{W}_n^\beta$ as follows:
\begin{equation}\label{l2inhomo}
\mathbb{E}\left[ \left( \bar{W}_n^\beta \right)^2 \right] = \sum_{h=0}^{n} d^{-h} e^{\beta^2 \sum_{i=1}^{h} f\left( \frac{i}{n} \right)}
\end{equation}
We can easily see from the expression above that the right hand side is monotonic in $\beta$, hence well-defined. In both cases in the assertion, we will prove the upper bound and lower bound of $\bar{\beta}_2$ respectively. 

For the lower bound: From the expression, we can see that each term in the series expression \eqref{l2inhomo} for $\mathbb{E}\left[ \left( \bar{W}_n^\beta \right)^2 \right]$ is bounded above by the corresponding term in the time-homogeneous case (i.e., the terms in the series expansion $\mathbb{E}\left[ \left( W_n^\beta \right)^2 \right]$). As a result, we have when $\beta<\beta_2$, $\mathbb{E}\left[ \left( \bar{W}_n^\beta \right)^2 \right]$ is uniformly bounded in $n$. 

For the upper bound: Fix $0 < \epsilon < 1$ and take $\beta > \sqrt{\log d}$. By the continuity of $f$, there exists $\delta > 0$ such that $f(x) < 1 - \delta$ for all $0 \leq x < \epsilon$. We then have:
\begin{equation*}
\mathbb{E}\left[ \left( \bar{W}_n^\beta \right)^2 \right] \geq \sum_{h=0}^{\lfloor \epsilon n \rfloor} d^{-h} e^{\beta^2 \sum_{i=1}^{h} f\left( \frac{i}{n} \right)}
\end{equation*}
\begin{equation*}
\geq \sum_{h=0}^{\lfloor \epsilon n \rfloor} Cd^{-h} e^{\beta^2 h (1 - \delta)} = \sum_{h=0}^{\lfloor \epsilon n \rfloor} Ce^{h \left( -\log d + \beta^2 (1 - \delta) \right)}
\end{equation*}
where $C$ is a constant depending only on $\beta$. For any fixed $\beta > \sqrt{\log d}$, we can choose $\epsilon>0$ sufficiently samll so that $\delta > 0$ becomes sufficiently small and satisfying $-\log d + \beta^2 (1 - \delta) > 0$, and the series diverges as $n \to \infty$.
\end{proof}

Prior to proceeding with the proofs, we state a key lemma that facilitates the transformation of Gaussian random variables across different measures. This result will play a central role in our proof (we refer readers to \cite{berestycki2025gaussianfreefieldliouville} , Lemma 2.5 for a proof): 

\begin{lemma} \label{Girsanov}(Cameron-Martin-Girsanov formula)
     Let $X = (X_1,...,X_n)$ be a Gaussian vector under the law $\mathbb{P}$, with mean $\mu$ and covariance matrix $V$. Let $\alpha\in\mathbb{R}^n$ be a vector and we define a new probability measure by:\begin{equation*}
         \frac{d\tilde{\mathbb{P}}}{d\mathbb{P}}=\frac{e^{\langle \alpha,X\rangle}}{\mathbb{E}[e^{\langle \alpha,X\rangle}]}
     \end{equation*}
     (here $\langle \alpha,X\rangle$ is the scalar product of two vectors $\alpha$ and $X$) Then under $\tilde{\mathbb{P}}$, $X$ is still a Gaussian vector, with covariance matrix $V$ and mean $\mu+V\alpha$.
\end{lemma}

  In other (more plain) words, suppose we tilt the law of a Gaussian vector $X=(X_1,...,X_n)$ by some linear functional $X\to \langle\alpha, X\rangle$. Then the process remains Gaussian, with unchanged covariances, however the mean is shifted, and the new mean of the random variable $X_i$ (which original law under $\mathbb{P}$ is $X_i\sim\mathcal{N}(\mu_i,\sigma_i^2)$, is:\begin{equation*}
      \mu_i'=\mu_i+\text{Cov}(X_i,\langle \alpha,X\rangle)
  \end{equation*}

Having the lemma stated, we proceed our proof of the $L^1$ convergence in the sub-critical regime.

\begin{theorem} [Lower Bound for $\bar{\beta}_c$ and Universality Behavior] \label{Main2}
\textsc{(Phase Transition and Universality for Time-inhomogeneous BRW)}
Let $f:[0,1]\to[0,1]$ be a continuous, non-increasing function with $f(0)=1$. Define the normalized partition function for time-inhomogeneous weights by:
\[
\bar{W}_n^\beta=\frac{1}{d^n}\sum_{|\mathbf{x}|=n} e^{\beta \bar{H}_n(\mathbf{x})-\frac{\beta^2}{2}\sum_{i=1}^n f\!\left(\frac{i}{n}\right)},
\]
where the term $\bar{H}_n$ defined by:
\[
\bar{H}_n(\mathbf{x})=\sum_{i=1}^n \sqrt{f\!\left(\frac{i}{n}\right)}w(x_i)
\]
is the Hamiltonian for a branch $\mathbf{x}$ with $|\mathbf{x}|=n$. Then the following universality in $L^1$ limit holds:
\[
\mathbb{E}\bigl[|W_n^\beta - \bar{W}_n^\beta|\bigr] \to 0
\quad\text{as }n\to\infty,\text{ for all }\beta<\beta_c=\sqrt{2\log d}.
\]
As a direct consequence, when $\beta<\beta_c$, $\bar{W}_n^\beta$ converges to an a.s. non-degenerate limit $\bar{W}_\infty^\beta$ in probability, which has the same law as $W_\infty^\beta$.
\end{theorem}

\begin{proof}
        We prove the $L^1$-convergence by dividing the branches into ``good environment'' and its complement, for both time-homogeneous and time-inhomogeneous environments. Using the Cameron-Martin-Girsanov formula mentioned previously, we can see that the branches in ``good environment'' in both cases dominates the total mass (i.e. the $L^1$-norm), and converges in $L^2$ to the same random variable.
        
        Let $\beta < \alpha < 2\beta$ be a fixed constant. For any branch $|\mathbf{x}| = n$, we define the ``good environment'' event for the time-homogeneous BRWs, denoted $G_n^\alpha(\mathbf{x})$ (roughly speaking, it corresponds to the environment where the Hamiltonian of a branch $\mathbf{x}$ remains ``controlled'') for all $T \in \{n_0, \ldots, n\}$ ,with $n_0 \geq 0$ a fixed integer, by:
\begin{equation*}
G_n^\alpha(\mathbf{x}) := \left\{ \forall T \in \{n_0, \ldots, n\}, \, H_T(\mathbf{x}) < \alpha T \right\}.
\end{equation*}

We define $J_n$ as the partition function restricted to good environments (since there is no ambiguity, we drop $\beta$ in the notation):
\begin{equation*}
J_n = \frac{1}{d^n} \sum_{|\mathbf{x}|=n} e^{\beta H_n(\mathbf{x}) - \frac{1}{2}\beta^2 n} \mathbf{1}_{\{\mathbf{x} \in G_n^\alpha(\mathbf{x})\}},
\end{equation*}
and $K_n$ as its complement relative to $W_n^\beta$:
\begin{equation*}
K_n = \frac{1}{d^n} \sum_{|\mathbf{x}|=n} e^{\beta H_n(\mathbf{x}) - \frac{1}{2}\beta^2 n} \mathbf{1}_{\{\mathbf{x} \in (G_n^\alpha(\mathbf{x}))^c\}}.
\end{equation*}
This gives the $L^1$-norm decomposition:
\begin{equation*}
\mathbb{E}\left[ W_n^\beta \right] = \mathbb{E}\left[ J_n \right] + \mathbb{E}\left[ K_n \right].
\end{equation*}

Analogously, we denote by $\bar{G}_n$ ,$\bar{J}_n$ and $\bar{K}_n$ the above quantities for time-inhomogeneous BRWs, which are given by:
\begin{equation*}
            \bar{G}_n^\alpha(\mathbf{x})=\{\forall T\in \{n_0,...,n\}, \bar{H}_T(\mathbf{x})<\alpha \sum_{i=1}^{T} f(\frac{i}{n})\}
        \end{equation*}
        and
        \begin{equation*}
            \bar{J}_n = \frac{1}{d^n} \sum_{|\mathbf{x}|=n} e^{\beta \bar{H}_n(\mathbf{x}) - \frac{\beta^2}{2} \sum_{i=1}^{n} f\left(\frac{i}{n}\right) w(x_i)}\mathbf{1}_{\{\mathbf{x}\in \bar{G}_n^\alpha(\mathbf{x})\}}
        \end{equation*}
        
        By the triangle inequality, we have:
        \begin{equation*}
            \mathbb{E}\left[|\bar{W}_n^\beta - W_n^\beta|\right] = \mathbb{E}\left[|J_n - \bar{J}_n + K_n - \bar{K}_n|\right] \leq \mathbb{E}\left[|J_n - \bar{J}_n|\right] + \mathbb{E}\left[K_n\right] + \mathbb{E}\left[\bar{K}_n\right].
        \end{equation*}
        It follows that when $n_0 \to \infty$, both $\mathbb{E}[K_n]$ and $\mathbb{E}[\bar{K}_n]$ tend to $0$ uniformly in $n$. To verify this, define the tilted probability measure $\hat{\mathbb{P}}$ by:
        \begin{equation*}
            \frac{\text{d}\hat{\mathbb{P}}}{\text{d}\mathbb{P}} = \frac{e^{\beta \bar{H}_n(\mathbf{x})}}{\mathbb{E}\left[e^{\beta \bar{H}_n(\mathbf{x})}\right]}.
        \end{equation*}
        By Cameron-Martin-Girsanov formula (Lemma \ref{Girsanov}), for any $T\in\{n_0,...,n\}$, under $\hat{\mathbb{P}}$, $\bar{H}_T(\mathbf{x})$ is a Gaussian random variable with mean $\beta \sum_{i=1}^{T} f\left(\frac{i}{n}\right)$ and preserves the covariance structure of $\bar{H}_T(\mathbf{x})$ under $\mathbb{P}$. Thus for any $|\mathbf{x}|=n$, we have the probability of the event $\hat{\mathbb{P}}\{\bar{G}_n^\alpha(\mathbf{x})\}$ equals to:\begin{equation*}
            \hat{\mathbb{P}}\{\bar{G}_n^\alpha(\mathbf{x})\}=\hat{\mathbb{P}}\{\forall T\in \{n_0,...,n\}, \bar{H}_T(\mathbf{x})<\alpha \sum_{i=1}^{T} f\left(\frac{i}{n}\right)\}
        \end{equation*}
        \begin{equation*}
           =\mathbb{P}\{\forall T\in \{n_0,...,n\}, \bar{H}_T(\mathbf{x})<(\alpha -\beta) \sum_{i=1}^{T} f\left(\frac{i}{n}\right)\} \geq \mathbb{P}\{\bar{G}_n^{\alpha-\beta}(\mathbf{x})\}
        \end{equation*}
        Consequently,
        \begin{equation*}
            \mathbb{E}[\bar{J}_n] = \hat{\mathbb{P}}\{\bar{G}_n^\alpha(\mathbf{x})\} \geq \mathbb{P}\{\bar{G}_n^{\alpha-\beta}(\mathbf{x})\} = 1 - \epsilon(n_0),
        \end{equation*}
        where $\epsilon(n_0) \to 0$ as $n_0 \to \infty$ uniformly in $n$. The probability that any single branch $\mathbf{x}$ with $|\mathbf{x}| = n$ does not lie in the good environment can be made arbitrarily small by choosing $n_0$ sufficiently large, with the upper bound uniform in $n$:
\begin{equation*}
\mathbb{P}\left\{ G_n^\alpha(\mathbf{x}) \right\} \geq 1 - \epsilon(n_0),
\end{equation*}
where $\epsilon$ depends only on $\alpha$ and $n_0$, and satisfies $\epsilon(n_0) \to 0$ as $n_0 \to \infty$ for each fixed $\alpha>\beta$. To verify this, we bound the probability of the complement event using Gaussian tail estimates:
\begin{equation}\label{good}
\mathbb{P}\left\{ \left(G_n^\alpha(\mathbf{x})\right)^c \right\} \leq \mathbb{P}\left\{ \exists T \geq n_0, \, H_T(\mathbf{x}) \geq \alpha \sum_{i=1}^{T} f\left(\frac{i}{n}\right) \right\} \leq \sum_{T \geq n_0} e^{-\alpha^2 \sum_{i=1}^{T} f\left(\frac{i}{n}\right)},
\end{equation}
which tends to $0$ as $n_0 \to \infty$ uniformly in $n$. A similar argument yields $\mathbb{E}[\bar{J}_n]=1-\bar{\epsilon}(n_0)$ where $\bar{\epsilon}(n_0) \to 0$ as $n_0 \to \infty$ uniformly in $n$. While the right-hand side of \eqref{good} converges for any $\alpha > 0$, we later use the good environment $G_n^{\alpha-\beta}(\mathbf{x})$ with $\alpha < 2\beta$, justifying our restriction to $\beta < \alpha < 2\beta$.

Next, we will show that when $\beta<\sqrt{2}\bar{\beta}_2$, $J_n$ and $\bar{J}_n$ are both uniformly bounded in $L^2$. Since we have seen in Theorem \ref{l2inhomo} that $\bar{\beta}_2 = \beta_2$, we have that $\bar{J}_n$ is uniformly bounded in $L^2$ when $\beta<\beta_c$. For pairs of branches $(\mathbf{x}, \mathbf{y})$ with $C(\mathbf{x}, \mathbf{y}) = h$, define the transformed measure $\tilde{\mathbb{P}}^h$ by:
         \begin{equation*}
         \frac{d\tilde{\mathbb{P}}}{d\mathbb{P}} = \frac{e^{\beta \left( \bar{H}_n(\mathbf{x}) + \tilde{\bar{H}}_n(\mathbf{x}) \right)}}{\mathbb{E}\left[ e^{\beta \left( \bar{H}_n(\mathbf{x}) + \tilde{\bar{H}}_n(\mathbf{x}) \right)} \right]},
         \end{equation*}
         and under $\tilde{\mathbb{P}}^h$, $\bar{H}_h(\mathbf{x})\sim \mathcal{N}(2\beta\sum_{i=1}^{h} f(\frac{i}{n}), \beta\sum_{i=1}^{h} f(\frac{i}{n}))$, so we have: \begin{equation}\label{J_n_bar}
\mathbb{E}\left[ J_n^2 \right] = \sum_{h=0}^n d^{-h} e^{\beta^2 \sum_{i=1}^{h} f(\frac{i}{n})} \tilde{\mathbb{P}}^h \left\{ \bar{G}_n^\alpha(\mathbf{x}) \cap \tilde{\bar{G}}_n^\alpha(\mathbf{x}) \right\}.
\end{equation}
\begin{equation}\label{l2bound}
    \leq \sum_{h=0}^n d^{-h} e^{\beta^2 \sum_{i=1}^{h} f(\frac{i}{n})} \tilde{\mathbb{P}}^h\{\bar{H}_h(\mathbf{x})<\alpha\sum_{i=1}^h f(\frac{i}{n})\}\leq \sum_{h=0}^n d^{-h} e^{(\beta^2-\frac{1}{2}(\alpha-2\beta)^2) \sum_{i=1}^{h} f(\frac{i}{n})}
\end{equation}
via Gaussian estimates. Finally, by letting $\alpha\to\beta$, we can see that the series is uniformly bounded in $n$ when $\beta<\sqrt{2}\bar{\beta_2}$, by a direct comparison with the $L^2$ computation of $\bar{W}_n^\beta$ in equation \ref{l2inhomo}. A similar comparison with the $L^2$ computation of $W_n^\beta$ yields that when $\beta<\sqrt{2}\beta_2=\beta_c$, $J_n$ is uniformly bounded in $L^2$. 

        It remains to show $J_n - \bar{J}_n \to 0$ in $L^1$, which we establish by proving that the convergence holds in $L^2$. We start by expressing the $L^2$-norm of $J_n - \bar{J}_n$ as
        \begin{equation}\label{difference2}
            \mathbb{E}\left[(\bar{J}_n^\beta - J_n^\beta)^2\right] = \mathbb{E}\left[(\bar{J}_n^\beta)^2 - \bar{J}_n^\beta J_n^\beta\right] + \mathbb{E}\left[(J_n^\beta)^2 - \bar{J}_n^\beta J_n^\beta\right].
        \end{equation}
        We will show that for any fixed $n_0 > 0$, both terms on the right-hand side of \eqref{difference2} tend to $0$ as $n \to \infty$.
        
        Fix $n_1 > 0$, and we decompose $\mathbb{E}\left[(\bar{J}_n^\beta)^2 - \bar{J}_n^\beta J_n^\beta\right]$ by $C(\mathbf{x},\mathbf{y})=h\in \{0,...,n\}$ into two sums over $h \leq n_1$ and $h > n_1$, so that $ \mathbb{E}\left[(\bar{J}_n^\beta)^2 - \bar{J}_n^\beta J_n^\beta\right]$ equals:
        \begin{equation*}
            \sum_{h=0}^{n_1} d^{-h} [\left( e^{\beta^2 \sum_{i=1}^{h} f\left(\frac{i}{n}\right)}\mathbb{P}{\{\bar{G}_n^\alpha(\mathbf{x}) \cap \bar{\tilde{G}}_n^\alpha(\mathbf{x})\}}\right) 
        \end{equation*}
        \begin{equation*}
            -\left(e^{\frac{1}{2}\beta^2 \left( \sum_{i=1}^{h} \left(1 + \sqrt{f\left(\frac{i}{n}\right)}\right)^2 - (h) - \sum_{i=1}^{h} f\left(\frac{i}{n}\right)\right)} \mathbb{P}{\{\bar{G}_n^\alpha(\mathbf{x}) \cap \bar{\tilde{G}}_n^\alpha(\mathbf{x})\}}\right) ]
        \end{equation*}
        \begin{equation*}
            +\sum_{h=n_1+1}^{n} [(d^{-h} \left( e^{\beta^2 \sum_{i=1}^{h} f\left(\frac{i}{n}\right)}\mathbb{P}{\{\bar{G}_n^\alpha(\mathbf{x}) \cap \bar{\tilde{G}}_n^\alpha(\mathbf{x})\}}\right) 
        \end{equation*}
        \begin{equation*}
            -\left(e^{\frac{1}{2}\beta^2 \left( \sum_{i=1}^{h} \left(1 + \sqrt{f\left(\frac{i}{n}\right)}\right)^2 - (h) - \sum_{i=1}^{h} f\left(\frac{i}{n}\right)\right)} \mathbb{P}{\{\bar{G}_n^\alpha(\mathbf{x}) \cap \bar{\tilde{G}}_n^\alpha(\mathbf{x})\}}\right)]
        \end{equation*}
             
        For $\beta < \beta_c$ and any fixed $n_1 > 0$, form Lemma \ref{JNConvergence} we can see that $\bar{J}_nJ_n$ and $\bar{J}_n^2$ converges to the same random variable $L^1$, since $\mathbb{E}[(\bar{J}_n)^2]$ and $\mathbb{E}[\bar{J}_n J_n]$ has the following expansions:
        \begin{equation*}
            \mathbb{E}\left[(\bar{J}_n)^2\right] = \sum_{h=0}^{n_1} d^{-h} e^{\beta^2 \sum_{i=1}^{h} f(\frac{i}{n})} g^\alpha(h) + F(n_1)
        \end{equation*}
        and 
        \begin{equation*}
            \quad \mathbb{E}\left[\bar{J}_n J_n\right] = \sum_{h=0}^{n_1} d^{-h} e^{\beta^2 \left( \sum_{i=1}^{h} \left(1 + \sqrt{f\left(\frac{i}{n}\right)}\right)^2 - (h) - \sum_{i=1}^{h} f\left(\frac{i}{n}\right)\right)} g^\alpha(h) + F'(n_1)
        \end{equation*}
        where $F(n_1)$ and $F'(n_1)$ can be arbitrarily small by choosing $n_1$ sufficiently large. Since $n_1 > 0$ is arbitrary, both $\mathbb{E}[(\bar{J}_n)^2]$ and $\mathbb{E}[\bar{J}_n J_n]$ converge to $\sum_{h=0}^\infty d^{-h} e^{\beta^2 h} g^\alpha(h)$ as $n \to \infty$.
        
        To confirm this, note that for any fixed $n_1 > 0$ and $h \in \{0, 1, \dots, n_1\}$, $\frac{h}{n} \to 0$ as $n \to \infty$; by continuity of $f$, this implies $f\left(\frac{h}{n}\right) \to f(0) = 1$ and $\sqrt{f\left(\frac{h}{n}\right)} \to 1$ as $n \to \infty$, and we have the convergence of $\mathbb{E}\left[(\bar{J}_n)^2\right]$ and $\mathbb{E}\left[\bar{J}_n J_n\right]$ by dominated convergence theorem. Thus, $\mathbb{E}\left[(\bar{J}_n^\beta)^2 - \bar{J}_n^\beta J_n^\beta\right] \to 0$ as $n \to \infty$. By an identical argument, $\mathbb{E}\left[(J_n^\beta)^2 - \bar{J}_n^\beta J_n^\beta\right] \to 0$ as $n \to \infty$. This completes the proof of the convergence $\bar{J}_n - J_n$ to $0$ in $L^2$, and hence the convergence $\bar{W}_n^\beta-W_n^\beta$ to $0$ in $L^1$.
\end{proof}

As mentioned in the proof, we establish the $L^2$-convergence of $J_n$ as introduced in the proof of Theorem \ref{Main2}. Together with the observation that $n_0 \geq 0$ may be chosen arbitrarily large (thereby making $K_n$ arbitrarily small in $L^1$), this implies that $(W_n^\beta)_{n\geq 0}$ converges in $L^1$ to $W_\infty^\beta$ within the subcritical regime $\beta < \sqrt{2\log d}$. By identical lines of reasoning, the same convergence result holds for time-inhomogeneous environments, which is $\bar{J}_n$ appearing in the proof of Theorem \ref{Main2}.
 
\begin{lemma} ($L^2$ Convergence of Partition Function in Good Environment) \label{JNConvergence}
    Under the notation of Theorem \ref{Main2}'s proof, $J_n$ converges in $L^2$ for each fixed $n_0 \geq 0$.
\end{lemma}

\begin{proof}
    We prove convergence by showing for each fixed $n_0\geq 0$, $J_n$ is a Cauchy sequence in $L^2$. For this, note that
    \begin{equation*}
        \mathbb{E}\left[(J_n - J_m)^2\right] = \mathbb{E}\left[J_n^2\right] + \mathbb{E}\left[J_m^2\right] - 2\mathbb{E}\left[J_n J_m\right].
    \end{equation*}
    Since $J_n$ is uniformly bounded in $L^2$, it suffices to establish two inequalities: an upper bound for $\mathbb{E}\left[J_n^2\right]$ and a lower bound for $\mathbb{E}\left[J_n J_m\right]$:
    \begin{equation}\label{ineq1}
        \limsup_{n \to \infty} \mathbb{E}\left[J_n^2\right] \leq \sum_{h=0}^{\infty} d^{-h} e^{\beta^2 h} g^\alpha(h)
    \end{equation}
    and
    \begin{equation}\label{ineq2}
        \liminf_{n,m \to \infty} \mathbb{E}\left[J_n J_m\right] \geq \sum_{h=0}^{\infty} d^{-h} e^{\beta^2 h} g^\alpha(h),
    \end{equation}
    where for each $h \geq 0$ (corresponding to $C(\mathbf{x},\mathbf{y}) = h$), $g^\alpha(h)$ is defined by:
    \begin{equation*}
        g^\alpha(h) = \tilde{\mathbb{P}}^h\left\{ G^\alpha(\mathbf{x}) \cap \tilde{G}^\alpha(\mathbf{x}) \right\} := \lim_{n \to \infty} \tilde{\mathbb{P}}^h\left\{ G_n^\alpha(\mathbf{x}) \cap \tilde{G}_n^\alpha(\mathbf{x}) \right\}.
    \end{equation*}
    Lemma \ref{GNConvergence} guarantees this convergence is uniform for $C(\mathbf{x},\mathbf{y}) = h < n_1$ and any fixed $n_1 \geq 0$.

    Our approach involves decomposing the series for $\mathbb{E}\left[J_n^2\right]$ and $\mathbb{E}\left[J_n J_m\right]$ into two parts: sums over pairs $(\mathbf{x},\mathbf{y})$ with $C(\mathbf{x},\mathbf{y}) = h < n_1$, and their complements. As shown previously, the former dominates the contribution to $\mathbb{E}\left[J_n^2\right]$ for large $n_1$ (with the complement being negligible). Additionally, the probability of being in the good environment converges to $g^\alpha(h)$ uniformly for $h < n_1$.

    For \eqref{ineq1}, we use the transformed measure $\tilde{\mathbb{P}}^h$ from the $L^2$-boundedness proof of Theorem \ref{Main2}. Fix $n_1 \in \{n_0, \ldots, n\}$ (eventually tending to infinity) and decompose the series by the number of differing steps between $\mathbf{x}$ and $\mathbf{y}$:
    \begin{equation*}
        \mathbb{E}\left[J_n^2\right] = \sum_{h=1}^{n_1} d^{-h} e^{\beta^2 h} \tilde{\mathbb{P}}^h\left\{ G_n^\alpha(\mathbf{x}) \cap \tilde{G}_n^\alpha(\mathbf{x}) \right\} + \sum_{h=n_1+1}^{n} d^{-h} e^{\beta^2 h} \tilde{\mathbb{P}}^h\left\{ G_n^\alpha(\mathbf{x}) \cap \tilde{G}_n^\alpha(\mathbf{x}) \right\}
    \end{equation*}
    \begin{equation*}
        \leq \sum_{h=1}^{n_1} d^{-h} e^{\beta^2 h} \tilde{\mathbb{P}}^h\left\{ G_n^\alpha(\mathbf{x}) \cap \tilde{G}_n^\alpha(\mathbf{x}) \right\} + F^\alpha(n_1),
    \end{equation*}
    where $F^\alpha(n_1)$ can be made arbitrarily small (uniformly in $n$) by choosing $n_1$ sufficiently large. As in Theorem \ref{Main2}'s proof, $F^\alpha(n_1)$ is bounded by the tail of the convergent series \eqref{l2bound}. Letting $n \to \infty$ and applying the dominated convergence theorem (with domination by the $L^2$-bound from \eqref{l2bound}) yields \eqref{ineq1}.

    For \eqref{ineq2}, consider $\mathbf{x}$ with $|\mathbf{x}| = m$, $\mathbf{y}$ with $|\mathbf{y}| = n$, and $C(\mathbf{x},\mathbf{y}) = h \leq n_1\leq \min\{m,n\}$. Define the transformed measure for differing-length walks (where $\tilde{H}_m(\mathbf{x})$ denotes the Hamiltonian of length-$m$ branch $\mathbf{y}$):
    \begin{equation*}
        \frac{d\tilde{\mathbb{P}}^h}{d\mathbb{P}} = \frac{e^{\beta \left( H_n(\mathbf{x}) + \tilde{H}_m(\mathbf{x}) \right)}}{\mathbb{E}\left[ e^{\beta \left( H_n(\mathbf{x}) + \tilde{H}_m(\mathbf{x}) \right)} \right]}.
    \end{equation*}
    The argument mirrors that of \eqref{ineq1}: uniform convergence from Lemma \ref{GNConvergence} provides a lower bound when restricting to $C(\mathbf{x},\mathbf{y}) < n_1$, with the complement becoming negligible as $n_1$ becomes large.

    Combining \eqref{ineq1} and \eqref{ineq2} shows $J_n$ is a Cauchy sequence in $L^2$, hence convergent in $L^2$ for each fixed $n_0 \geq 0$.
\end{proof}

The following lemma, which was used in the proof of the previous Lemma, establishes the uniform convergence of the events under the tilted probability: 

\begin{lemma} (Convergence of Joint Probability in Good Environment) \label{GNConvergence}
    Let $n_1 > 0$ be a fixed integer. For all $h \in \{0, \ldots, n_1\}$, the convergence
    \begin{equation*}
        \lim_{m,n\to\infty}\tilde{\mathbb{P}}^h\left\{ G_n^\alpha(\mathbf{x}) \cap G_m^\alpha(\mathbf{y}) \right\}= g^\alpha(h) := \mathbb{P}\left\{\tilde{G}^{\alpha,h}(\mathbf{x}) \cap \tilde{G}^{\alpha,h}(\mathbf{y}) \right\}
    \end{equation*}
    uniformly for all $\mathbf{x}$ with $|\mathbf{x}| = n$, $\mathbf{y}$ with $|\mathbf{y}| = m$, and $C(\mathbf{x},\mathbf{y}) = h$. Here, $\tilde{G}^{\alpha,h}$ is defined by
    \begin{equation*}
        \tilde{G}^{\alpha,h}(\mathbf{x}) = \left\{ \forall T \geq n_0, \, \tilde{H}^h_T(\mathbf{x}) < \alpha T \right\},
    \end{equation*}
    where $\tilde{H}^h_T$ denotes the tilted Hamiltonian (i.e., the original Hamiltonian $H_T$ under the transformed measure $\tilde{\mathbb{P}}^h$).
    Moreover, the same convergence (to the same limit) holds for time-inhomogeneous weights:\begin{equation*}
        \lim_{m,n\to\infty}\tilde{\mathbb{P}}^h\left\{ \bar{G}_n^\alpha(\mathbf{x}) \cap \bar{G}_m^\alpha(\mathbf{y}) \right\}= g^\alpha(h) := \mathbb{P}\left\{\tilde{G}^{\alpha,h}(\mathbf{x}) \cap \tilde{G}^{\alpha,h}(\mathbf{y}) \right\}
    \end{equation*} 
\end{lemma}

\begin{proof}
    For each $\mathbf{x}$ with $|\mathbf{x}| = n$ and $T \in \{0, \ldots, n\}$, let $E^\alpha(T)$ to be the event $E^\alpha_T(\mathbf{x}) = \left\{ H_T(\mathbf{x}) < \alpha T \right\}$ and define $\tilde{E}^{\alpha,h}_T(\mathbf{x})$ analogously for $\tilde{H}^h_T(\mathbf{x})$. By definition, $\tilde{G}^{\alpha,h} = \bigcap_{T=n_0}^\infty \tilde{E}^{\alpha,h}_T(\mathbf{x})$.

    We first establish the upper bound:
    \begin{equation}\label{upper}
        \limsup_{m,n \to \infty} \tilde{\mathbb{P}}^h\left\{ G_n^\alpha(\mathbf{x}) \cap G_m^\alpha(\mathbf{y}) \right\} \leq g^\alpha(h).
    \end{equation}
    To see this, observe that for any fixed $n_2 > n_0$, when $m, n > n_2$, we have
    \begin{equation*}
        \tilde{\mathbb{P}}^h\left\{ G_n^\alpha(\mathbf{x}) \cap G_m^\alpha(\mathbf{y}) \right\} \leq \mathbb{P}\left\{ \bigcap_{T=n_0}^{n_2} \tilde{E}^h_T(\mathbf{x}) \cap \tilde{E}^h_T(\mathbf{y}) \right\}.
    \end{equation*}
    Since $n_2 > n_0$ is arbitrary, taking $n_2 \to \infty$ gives
    \begin{equation*}
        \limsup_{m,n \to \infty} \tilde{\mathbb{P}}^h\left\{ G_n^\alpha(\mathbf{x}) \cap G_m^\alpha(\mathbf{y}) \right\} \leq \tilde{\mathbb{P}}^h\left\{ \bigcap_{T=n_0}^\infty E^{\alpha,h}_T(\mathbf{x}) \cap E^{\alpha,h}_T(\mathbf{y}) \right\} = \tilde{\mathbb{P}}^h\left\{ G^\alpha(\mathbf{x}) \cap G^\alpha(\mathbf{y}) \right\} = g^\alpha(h),
    \end{equation*}
    confirming \eqref{upper}.

    For the lower bound, we use a similar argument to that in Theorem \ref{Main2}. For any fixed $\epsilon > 0$, there exists $n_2$ (independent of $m, n$) such that for all $m, n > 0$,
    \begin{equation*}
        \tilde{\mathbb{P}}^h\left\{ \bigcup_{T=n_2}^{\max\{m,n\}} \left( H_T(\mathbf{x}) \geq \alpha T \right) \right\} < \epsilon,
    \end{equation*}
    with the empty event assigned probability $0$.  The key is to observe that the constraint $H_n(\mathbf{x})<\alpha n$ defining the event $E_n(\mathbf{x})$ is essentially guaranteed for large $n$ under $\tilde{\mathbb{P}}^h$: when $C(\mathbf{x},\mathbf{y})<n_1$ so the ranches $\mathbf{x}$ and $\mathbf{y}$ are well separated (note that this would be false however if $\mathbf{x}$ and $\mathbf{y}$ would not be separated: for instance if $\mathbf{x}=\mathbf{y}$, then $\tilde{H}_n(\mathbf{x})$ would be of the order of $2\beta n$, which is larger than $\alpha n$.) As a result, this follows directly from the proof of Theorem \ref{Main2} (where $\tilde{\mathbb{P}}\{G_n^\alpha(\mathbf{x})\} \geq 1 - \epsilon(n_0)$) by replacing $n_0$ with $n_2$.

    Consequently, the lower bound holds:
    \begin{equation}\label{lower}
        \tilde{\mathbb{P}}^h\left\{ G_n^\alpha(\mathbf{x}) \cap G_m^\alpha(\mathbf{y}) \right\} \geq \mathbb{P}\left\{ \bigcap_{T=n_0}^{n_2} \tilde{E}^{\alpha,h}_T(\mathbf{x}) \cap \tilde{E}^{\alpha,h}_T(\mathbf{y}) \right\} - 2\epsilon.
    \end{equation}

    Letting $m, n \to \infty$ yields
    \begin{equation*}
        \liminf_{m,n \to \infty} \tilde{\mathbb{P}}^h\left\{ G_n^\alpha(\mathbf{x}) \cap G_m^\alpha(\mathbf{y}) \right\} \geq \mathbb{P}\left\{ \bigcap_{T=n_0}^{n_2} \tilde{E}^{\alpha,h}_T(\mathbf{x}) \cap \tilde{E}^{\alpha,h}_T(\mathbf{y}) \right\} - 2\epsilon 
    \end{equation*}
    \begin{equation*}
        \geq \mathbb{P}\left\{ \bigcap_{T=n_0}^\infty \tilde{E}^{\alpha,h}_T(\mathbf{x}) \cap \tilde{E}^{\alpha,h}_T(\mathbf{y}) \right\} - 2\epsilon
    \end{equation*}
    Since $\epsilon > 0$ is arbitrary, this establishes \eqref{lower}. Combining \eqref{upper} and \eqref{lower} completes the proof, with uniformity in $C(\mathbf{x},\mathbf{y}) \leq n_1$ preserved throughout.

    The uniform convergence to the same limit also holds for time-inhomogeneous weights. When $f$ is decreasing, for fixed $n_2 \geq 0$, when $n$ is large enough (such that $\frac{n_2}{n} \to 0$ and $f\left(\frac{h}{n}\right) \to 1$ for all $h \in \{0, \dots, n_2\}$), the probability of the following finite intersection of events converges:
\begin{equation*}
\mathbb{P}\left\{ \bigcap_{T=n_0}^{n_2} \tilde{\bar{E}}^h_T(\mathbf{x}) \cap \tilde{\bar{E}}^h_T(\mathbf{y}) \right\} \to \mathbb{P}\left\{ \bigcap_{T=n_0}^{n_2} \tilde{E}^h_T(\mathbf{x}) \cap \tilde{E}^h_T(\mathbf{y}) \right\},
\end{equation*}
as $n \to \infty$, by continuity of Gaussian cumulative density function. Here, $\tilde{\bar{E}}^h_T(\mathbf{x})$ is defined as $\tilde{\bar{E}}^h_T(\mathbf{x}) := \left\{ \bar{H}_T < \alpha \sum_{i=1}^T f\left(\frac{i}{n}\right) \right\}$.

As a result, we obtain the same upper bound since $n_2\geq 0$ can be chosen arbitrarily. The lower bound, meanwhile, can be derived via the same argument to that used for the time-homogeneous case in the proof above.\end{proof}

\begin{remark}
     Indeed, we can see from the proof that monotone decreasing and continuity is not actually relied in the proof in this section: such convergence holds as long as $f$ is continuous at $0$, satisfying $f(1)=1$ and $f(x)\in [0,1]$ for all $x\in [0,1]$. This observation will be used in the proof in CREM model in Section \ref{CREM}, where the covariance structure could be more arbitrary.
\end{remark}

\subsubsection{Upper Bound for Time-inhomogeneous Critical Parameter}

The following proposition establishes the upper bound for $\bar{\beta}_c$, i.e.  when $\beta > \beta_c$, $\bar{W}_n^\beta$ converges to $0$ in probability. Together with Theorem \ref{Main2}, this confirms that the normalized partition function $\bar{W}_n^\beta$ exhibits distinct phase transitions in its non-degeneracy of in-probability limit. Furthermore, the critical parameter $\bar{\beta}_c$ is defined by: \begin{equation*}
    \bar{\beta}_c := \sup\left\{\beta \geq 0 \mid \mathbb{P}\{\bar{W}_\infty^\beta = 0\} > 0\right\}
\end{equation*}
exists and coincides with that for homogeneous weights, i.e., $\beta_c = \bar{\beta}_c$. This result underscores the insensitivity to environmental variations, indicating that the macroscopic behavior of the partition function is dominated by the ``early generations'' of the branching structure, rather than being affected by prescribed variations in the branch weights.

\begin{theorem} [Upper Bound for $\bar{\beta}_c$] \label{Inhomogeneous upper bound}
  Let $\beta > \sqrt{2\log d}$ and let $f: [0,1] \to [0,1]$ be a continuous decreasing function satisfying $f(0) = 1$ and $f(1) = 0$. Then the normalized partition function $\bar{W}_n^\beta$ converges in probability to the random variable $\bar{W}_\infty^\beta \equiv 0$ as $n \to \infty$. As a consequence, $\bar{\beta}_c := \sup\left\{\beta \geq 0 \mid \mathbb{P}\{\bar{W}_\infty^\beta = 0\} > 0\right\}$ exists and coincides with that for homogeneous weights, i.e., $\beta_c = \bar{\beta}_c$.
\end{theorem}

\begin{proof}
    We fix a constant $\alpha \in (0,1)$ and it suffices to show that $\mathbb{E}\left[(\bar{W}_n^\beta)^\alpha\right] \to 0$, as this implies $\bar{W}_n^\beta$ converges in probability to $\bar{W}_\infty^\beta \equiv 0$.

    First, note that by the independence of weight increments, the partition function $\bar{W}_n^\beta$ can be decomposed as:
    \begin{equation*} 
        \bar{W}_n^\beta = \sum_{\vert \mathbf{x} \vert = n_1} d^{-n_1} e^{-\frac{1}{2}\beta^2 \sum_{i=1}^{n_1} f\left(\frac{i}{n}\right) + \beta \bar{H}_{n_1}(\mathbf{x})} \sum_{\mathbf{y}\in D^\mathbf{x}_{n-n_1}} d^{-(n - n_1)} e^{-\frac{1}{2}\beta^2 \sum_{i = n_1+1}^n f\left(\frac{i}{n}\right) + \beta \bar{H}_{n_1\to n}(\mathbf{y})},
    \end{equation*}
    where $D^{\mathbf{x}}_{n-n_1}$ denotes the set of offspring in the sub-GW tree rooted at a branch $\mathbf{|x|}=n_1$ with $n-n_1$ generations, and $\bar{H}_{n_1\to n}(\mathbf{x}) = \sum_{i = n_1 }^n \sqrt{f\left(\frac{i}{n}\right)} w(x_i)$ is the Hamiltonian from generation $n_1$ to $n$.

    We define a transformed (random) probability measure $\bar{\mathbb{P}}^{n_1}$ by (where $\mathbf{P}^*$ is the survival probability of GW tree):
    \begin{equation}\label{Random measure}
        \frac{\mathrm{d}\bar{\mathbb{P}}^{n_1}(\mathbf{x})}{\mathrm{d}\mathbf{P}^*} = \frac{e^{-\frac{1}{2}\beta^2 \sum_{i=1}^{n_1} f\left(\frac{i}{n}\right) + \beta \bar{H}_{n_1}(\mathbf{x})}}{\bar{W}_n^{n_1}}
    \end{equation}
    Let $\mathbf{E}_{n_1}$ denote the expectation under $\bar{\mathbb{P}}^{n_1}$, and abbreviate the decomposed partition function by:
    \[
        \bar{W}_n^{n_1} = \sum_{\vert \mathbf{x} \vert = n_1} d^{-n_1} e^{-\frac{1}{2}\beta^2 \sum_{i=1}^{n_1} f\left(\frac{i}{n}\right) + \beta \bar{H}_{n_1}(\mathbf{x})}
    \]
    and
    \[
        \tilde{\bar{W}}_n^{n_1} = \sum_{\mathbf{y}\in D^\mathbf{x}_{n-n_1}} d^{-(n - n_1)} e^{-\frac{1}{2}\beta^2 \sum_{i = n_1 + 1}^n f\left(\frac{i}{n}\right) + \beta \bar{H}_{n_1\to n}(\mathbf{y})}.
    \]
    For any function $F:D_{n_1}\to\mathbb{R}$ (recall that $D_{n_1}$ is the set of branches $|\mathbf{x}|=n_1$), we have
    \begin{equation*}
        \mathbf{E}_{n_1}[F(\mathbf{x})] = \frac{1}{\bar{W}_n^{n_1}} \sum_{\vert \mathbf{x} \vert = n_1} \frac{1}{d^{n_1}} e^{\beta \bar{H}_{n_1}(\mathbf{x}) - \frac{1}{2}\beta^2 \sum_{i=1}^{n_1} f\left(\frac{i}{n}\right)} F(\mathbf{x}),
    \end{equation*}
    which is $\mathcal{F}_{n_1}$-measurable (where $\mathcal{F}_{n_1}$ is the $\sigma$-field generated by the GW tree and the weights up to generation $n_1$). 

    It follows that:
    \begin{equation*}
        \mathbb{E}\left[(\bar{W}_n^\beta)^\alpha\right] = \mathbb{E}\left[(\bar{W}_n^{n_1})^\alpha \left(\mathbf{E}_{n_1}\left[\tilde{\bar{W}}_n^{n_1}\right]\right)^\alpha\right].
    \end{equation*}
    By the Cauchy-Schwarz inequality, this implies:
    \begin{equation}\label{Fractional moment}
        \mathbb{E}\left[(\bar{W}_n^\beta)^\alpha\right] \leq \mathbb{E}\left[(\bar{W}_n^{n_1})^{2\alpha}\right]^{\frac{1}{2}} \mathbb{E}\left[\left(\mathbf{E}_{n_1}\left[\tilde{\bar{W}}_n^{n_1}\right]\right)^{2\alpha}\right]^{\frac{1}{2}}.
    \end{equation}

    To bound the second term on the right hand side in \eqref{Fractional moment}, for $0 < \alpha < \frac{1}{2}$, Jensen's inequality implies that:    \begin{equation*}
        \mathbb{E}\left[\left(\mathbf{E}_{n_1}\left[\tilde{\bar{W}}_n^{n_1}\right]\right)^{2\alpha}\right] \leq \mathbb{E}\left[\mathbf{E}_{n_1}\left[\tilde{\bar{W}}_n^{n_1}\right]\right]^{2\alpha}
    \end{equation*}
    By the independence of $\tilde{\bar{W}}_n^{n_1}$ and the $\sigma$-field $\mathcal{F}_t$ (where the tilted probability measure $\bar{\mathbb{P}}^{n_1}$ is $\mathcal{F}_t$-measurable), we have the expectation are the same under two probability measures:
    \begin{equation*}
        \mathbb{E}\left[\mathbf{E}_{n_1}\left[\tilde{\bar{W}}_n^{n_1}\right]\right]^{2\alpha} = \mathbb{E}\left[\mathbb{E}\left[\tilde{\bar{W}}_n^{n_1}\right]\right]^{2\alpha}=1
    \end{equation*}

    We will bound $\mathbb{E}\left[(\bar{W}_n^{n_1})^{2\alpha}\right]$ by:
    \begin{equation*}
        \mathbb{E}\left[(\bar{W}_n^{n_1})^{2\alpha}\right] \leq d^{n_1} \mathbb{E}\left[d^{-2\alpha n} e^{-\alpha\beta^2 \sum_{i=1}^{n_1} f\left(\frac{i}{n}\right) + 2\alpha\beta \bar{H}_{n_1}(\mathbf{x})}\right]
    \end{equation*}
    \begin{equation*}
        \leq d^{(1 - 2\alpha)n_1} e^{-\alpha\beta^2 \sum_{i=1}^{n_1} f\left(\frac{i}{n}\right)} \mathbb{E}\left[e^{2\alpha\beta \bar{H}_{n_1}(\mathbf{x})}\right]
    \end{equation*}
    \begin{equation}\label{super inhomo inequality}
        \leq d^{(1 - 2\alpha)n_1} e^{-\alpha\beta^2 \sum_{i=1}^{n_1} f\left(\frac{i}{n}\right)} e^{2\alpha^2\beta^2 n_1}=e^{(\log d (1 - 2\alpha)-\frac{\alpha\beta^2}{n_1} \sum_{i=1}^{n_1} f\left(\frac{i}{n}\right)+2\alpha^2\beta^2)n_1}.
    \end{equation}

    Finally, for any fixed $n_1 > 0$ and $\beta > \sqrt{2\log d}$, we have when letting $n\to\infty$:\begin{equation*}
        (1 - 2\alpha)\log d  - \frac{\alpha\beta^2}{n_1} \sum_{i=1}^{n_1} f\left(\frac{i}{n}\right) + 2\alpha^2\beta^2\to (1 - 2\alpha)\log d - \alpha\beta^2 + 2\alpha^2\beta^2
    \end{equation*} 
    This convergence holds because $f\left(\frac{i}{n}\right)$ can be arbitrarily close to $1$ for all $i < n_1$ when $n$ becomes sufficiently large. Subsequently, by letting $\alpha$ approach $\frac{1}{2}$ arbitrarily close, we have $\log d \cdot (1 - 2\alpha) - \alpha\beta^2 + 2\alpha^2\beta^2 < 0$. Moreover, the right-hand side of \eqref{super inhomo inequality} can be made arbitrarily small by choosing $n_1$ sufficiently large. Combining with the inequality \eqref{Fractional moment}, we conclude that $\bar{W}_n^\beta \to 0$ in probability when $\beta > \beta_c = \sqrt{2\log d}$.
\end{proof}

\subsection{Applications: Time-homogeneous BRWs and Continuous Random Energy Model}

In this section, we will see that the framework discussed in Section \ref{Sub-Critical} could be naturally applied to two related models: the time-homogeneous BRWs and the Continuous random energy model (CREM).

\subsubsection{Biggins Martingale Convergence: Theorem \ref{Main1}}

Using the strategy of the proof of these aforementioned results, we can provide a novel proof for Theorem \ref{Main1}. Since the proof follows similar lines to that of Theorem \ref{Main2}, we prefer not to state all the details. Indeed, we can see that the proof is much simpler compared to the proof of Theorem \ref{Main2}: since $(W_n^\beta)_{n\geq 0}$ defines a unit-mean martingale, to show $L^1$ convergence it is sufficient to show that it is uniformly integrable (i.e., it is suffice to show that $J_n$ is uniformly bounded in $L^2$ for each fixed $n_0>0$). The upper bound follows from the same argument as the last part of the proof of Proposition \ref{Inhomogeneous upper bound}: for $\alpha\in (0,1)$ we have:
\begin{equation*}
\mathbb{E}\left[(W_n^\beta)^\alpha\right]
\leq d^n \mathbb{E}\left[d^{-\alpha n} e^{-\frac{1}{2}\alpha\beta^2 n + \alpha\beta H_n(\mathbf{x})}\right]= e^{n\left((1-\alpha)\log d - \frac{1}{2}\alpha\beta^2 + \frac{1}{2}\alpha^2\beta^2\right)}.
\end{equation*}
We can see that if $\beta$ satisfies $\log d - \frac{1}{2}\alpha\beta^2 < 0$, then $\mathbb{E}\left[(W_n^\beta)^\alpha\right] \to 0$ as $n \to \infty$, and we may choose $\alpha$ arbitrarily close to $1$ so that it holds for any $\beta>\sqrt{2\log d}$. 

\subsubsection{Convergence of the Partition Function of CREM: Proposition \ref{CREMPARFUN}}\label{CREM}

Here, we will provide an extension of our result to the continuous random energy model (CREM), see Section \ref{modelBRWs} for the precise definition of this model. 

In \cite{BOVIER2004481}, Bovier and Kurkova showed that there exists a critical parameter $\beta_c=\frac{\sqrt{2\log 2}}{\sqrt{\hat{A}'(0)}}$ such that the following phase transition occurs: a) For all $\beta< \beta_c$, the main contribution to the partition function comes from an exponential number of the particles. b) For all $\beta>\beta_c$, the maximum starts to contribute significantly to the partition function. The quantity $\beta_c$ is sometimes referred to as the static critical inverse temperature of the CREM, and the coincidence with $\bar{\beta}_c$ (if we normalize $A$ so that $\hat{A}'(0)=1$) in the present paper further suggests the potential relationship between these two models. Under the framework of the present paper, we can see that although the exact approximation terms are different in this setting, the Gaussian increment and the covariance structure allow us to prove the following:

\begin{proposition} (Convergence of CREM Partition Function)
    Suppose for any $x\in [0,1]$, $A'(x)$ exists almost everywhere, satisfies $\sup_{x\in [0,1]}A'(x)\leq A(0)$ and continuous at $0$ (i.e., $A\in  C^0[0,1]\cap C^1(0)$). Then when $\beta<\beta_c=\frac{\sqrt{2\log 2}}{\sqrt{A'(0)}}$, we have $Z_n^\beta$ converges to $Z_\infty^\beta$ in $L^1$ and $Z_\infty^\beta$ has the same law as $W_\infty^\beta$, and when $\beta>\beta_c$, we have $Z_n^\beta$ converges to $0$ in probability.
\end{proposition}

\begin{proof} (Sketch)
    Indeed, we can easily see that the proof follows similar lines to that of Theorem \ref{Main2}, so we only sketch the proof here. Under the binary tree at generation $n$ and its covariance structure, for each branch $|\mathbf{x}|=n$, and any $0\leq k\leq N$, we can see that the Hamiltonian of $u$ at generation $k$ has law $\mathcal{N}\left(0,NA(\frac{k}{N})\right)$. As a result, the $L^2$ norm of $Z_n^\beta$ is given by:
\begin{equation*}
   \mathbb{E}\left[ \left( Z_n^\beta \right)^2 \right] = \sum_{h=0}^{n} 2^{-h} e^{\beta^2 nA(\frac{h}{N})}
\end{equation*}
Since $A'(x)$ is continuous at $x=0$, by applying a similar argument to that in Theorem \ref{L^2} we can see that it is uniformly bounded when $\beta<\frac{\sqrt{\log 2}}{\sqrt{A'(0)}}$. Then we may apply the same measure tilting argument and obtain the exact law of the Hamiltonian $H_n(\mathbf{x})$ (where for each $T\in \{0,\dots,n\}$ and $|\mathbf{x}|=T$, $H_T(\mathbf{x})\sim \mathcal{N}(0,A(\frac{T}{N}))$) under the tilted measure, similarly to that in Theorem \ref{Main2}, and the rest of the proof follows since the continuity of $A'$ at $0$ implies that for any finite $n_0>0$, we have $A(\frac{T}{N})\to A(0)$ for any $0\leq T\leq n_0$. 

The upper bound for $\beta_c$ follows from the same lines as in Theorem \ref{Inhomogeneous upper bound}.
\end{proof}

\subsection{Behavior of Partition Function at Criticality}\label{Critical}

For a long time—over two decades after the establishment of the phase transition of Biggins martingale—the exact behavior of the time-homogeneous partition function $W_n^\beta$ at the critical value $\beta = \beta_c$ remained unknown, except that it lies in the strong disorder regime (i.e. $W_n^\beta$ converges to 0 almost surely). Naturally, this raises the question of the decay rate of $\bar{W}_n^\beta$ at criticality. Since $\beta_c = \bar{\beta_c}$ (as shown in Theorem \ref{Main2}), we will henceforth denote both by $\beta_c$ for the remainder of this section.  

The problem of determining the exact asymptotic behavior at criticality was resolved in the influential work of Hu and Shi \cite{HS09}, who established the following key result (Theorem 1.3 in \cite{HS09}):  
\begin{equation*}
    W_n^{\beta_c} = n^{-\frac{1}{2} + o(1)} \quad \text{a.s.}
\end{equation*}  

Notably, this implies that even though $\beta = \beta_c$ lies in the strong disorder regime, the partition function decays only polynomially. The next theorem suggests that this asymptotic behavior extends to time-inhomogeneous weights as well, and that the universality principle extends to the critical regime (though convergence is now almost sure). 

\begin{theorem} [Asymptotics of partition function at criticality]
    At the critical parameter $\beta_c$, we have $\bar{W}_n^{\beta_c}=n^{-\frac{1}{2}+o(1)}$ \text{almost surely}. 
\end{theorem}

\begin{proof}
    
    First, we will see that it is suffice to show the same asymptotics holds for all fractional moments for $\alpha\in (0,1)$. The arguments here are fairly standard, and variations of these arguments are also employed in related works \cite{HS09} and \cite{alberts2012nearcriticalscalingwindowdirected}:
    
    \begin{lemma}(Upper bounds)
        Let $(a_n)_{n \in \mathbb{N}}$ be a sequence of real numbers such that $|a_n| \to \infty$ as $n \to \infty$. Suppose that for every $\alpha \in (0,1)$, we have:
        \begin{equation*}
            \mathbb{E}[(\bar{W}_n^{\beta_c})^{\alpha}] = e^{\alpha a_n(1+ o(1))}.
        \end{equation*}
        Then $\bar{W}_n^{\beta_c} \leq e^{a_n (1 + o(1))}$ in probability as $n \to \infty$. For our purposes, we set $a_n=-\frac{1}{2}\log n$. Moreover, if for any $\epsilon > 0$, $\sum_{n \geq 1} e^{- \epsilon |a_n|} < \infty$, then $\bar{W}_n^{\beta_c}\leq e^{a_n (1 + o(1))}$ almost surely.
    \end{lemma}

    \begin{proof}
        
        Fix $\epsilon > 0$ and let $\alpha \in (0,1)$. By the given assumption, $\mathbb{E}[(\bar{W}_n^{\beta_c})^{\alpha}] \leq e^{\alpha a_n + \epsilon \alpha |a_n| /2}$ for all sufficiently large $n$. Applying Markov's inequality, we obtain:
        \begin{align*}
        \mathbb{P}(\bar{W}_n^{\beta_c} > e^{a_n + \epsilon |a_n|})=\mathbb{P}((\bar{W}_n^{\beta_c})^\alpha > e^{\alpha a_n + \epsilon \alpha|a_n|}) &\leq e^{-\alpha a_n - \epsilon \alpha |a_n|} \mathbb{E}[(\bar{W}_n^{\beta_c})^{\alpha}] \\
        & \leq e^{-\epsilon \alpha |a_n|/2}
        \end{align*}
        By the assumption that $|a_n| \to \infty$, it follows that $\bar{W}_n^{\beta_c} \leq e^{a_n(1 + o(1))}$ in probability. The almost sure convergence of the statement follows from Borel-Cantelli lemma.
    \end{proof}

    \begin{lemma}[Lower bounds]
        Let $(a_n)_{n \in \mathbb{N}}$ be a sequence satisfying $(\log n)^\frac{1}{2} \ll |a_n| \ll n$. Assume that for all $\alpha \in (0,1)$, we have:
        \[ \mathbb{E}[(\bar{W}_n^{\beta_c})^\alpha ]  = e^{\alpha a_n(1 + o(1))} .\]
        Then almost surely,
        \[ \bar{W}_n^{\beta_c}  \geq e^{a_n(1+o(1))} . \]
    \end{lemma}

    \begin{proof}
        We abbreviate the notation by defining the ``partial free energy'' for the time-inhomogeneous weight as (recall the definition of $\bar{H}_n^{\tau_n}(w)$ in the proof of Theorem \ref{Inhomogeneous upper bound}):
        \[ \bar{V}_{\tau_n}^n(w)=\beta_c\bar{H}_{n}^{\tau_n}(w)-\frac{1}{2}\beta_c^2\sum_{i=1}^{\tau_n}f(\frac{i}{n})-\tau_n\log d \]
        Fix $\epsilon > 0$, and we will show that the probability that $\mathbb{E}[(\bar{W}_n^{\beta_c})^\alpha]$ is beyond the rate $e^{\alpha a_n \pm \epsilon\alpha |a_n|}$ tends to 0 as $n\to\infty$. By the given assumption of the asymptotic rate of fractional moment, for any $\alpha \in (0,\frac{1}{2})$, we have:
        \[ \mathbb{E}[(\bar{W}_n^{\beta_c})^\alpha] \geq e^{\alpha a_n - \frac{\epsilon}{4}\alpha |a_n|} \]
        and
        \[ \mathbb{E}[ (\bar{W}_n^{\beta_c})^{2\alpha} ] \leq e^{2\alpha a_n + \frac{\epsilon}{4}|a_n|} \]
        for all sufficiently large $n$. Applying the Paley-Zygmund inequality, we derive:
        \begin{align}
        \mathbb{P} \left\{  \bar{W}_n^{\beta_c} > e^{a_n-\epsilon|a_n|}  \right\} &\geq \left( 1 - \frac{e^{\alpha ( a_n- \epsilon|a_n|)}}{\mathbb{E}[ (\bar{W}_n^{\beta_c})^\alpha]} \right)^2 \frac{\mathbb{E}[(\bar{W}_n^{\beta_c})^\alpha]^2}{\mathbb{E}[ (\bar{W}_n^{\beta_c})^{2\alpha}]} \notag \\
         &\geq \left( 1 -  e^{-3\epsilon\alpha|a_n|/4} \right)^2 e^{-\frac{3\epsilon}{4}|a_n|} \geq e^{-\epsilon|a_n|} \, , \label{rough_bound}
        \end{align}
        for all sufficiently large $n$. Now define $\tau_n$ to be:
        \[ \tau_n = \lceil\frac{2\epsilon|a_n|}{\log d}\rceil \]
        so that $\tau_n \ll n$ for all sufficiently large $n$. We then have the decomposition:
        \begin{align*}
        \bar{W}_n^{\beta_c} &= \sum_{|w| = \tau_n} \frac{1}{d^{\tau_n}}e^{\beta_c\bar{H}_n^{\tau_n}(w)-\frac{1}{2}\beta_c^2\sum_{i=1}^{\tau_n}f(\frac{i}{n})} \!\!\!\! \sum_{\substack{v \in D_{n-\tau_n}(w)}} \frac{1}{d^{n-\tau_n}}e^{\beta_c\bar{H}_{\tau_n\to n}(w)-\frac{1}{2}\beta_c^2\sum_{i=\tau_n+1}^{n}f(\frac{i}{n})} \\
        &\geq \exp \big\{ \max_{|w| = \tau_n} \bar{V}_{\tau_n}^n(w) \big \} \sum_{|w|=\tau_n} \!\!\!\ \sum_{\substack{v \in D_{n-\tau_n}(w)}} e^{\beta_c\bar{H}_{\tau_n\to n}(w)-\frac{1}{2}\beta_c^2\sum_{i=\tau_n+1}^{n}f(\frac{i}{n})}
        \end{align*}
        We denote the rightmost sum by $Y_{n-\tau_n}(w)$. The above inequality then yields the following estimate:
        \begin{align*}
        \mathbb{P} \left\{\bar{W}_n^{\beta_c} \leq e^{a_n - \epsilon |a_n|} \exp\{ \max_{|v|=\tau_n} \bar{V}_{\tau_n}^n (v)\} \right\} & \leq \mathbb{P} \left\{ \sum_{|w| = \tau_n} Y_{n-\tau_n}(v) \leq e^{a_n - \epsilon|a_n|} \right\} \\
         & \leq \mathbb{P} \left\{ Y_{n-\tau_n} (v)\leq e^{a_n - \epsilon|a_n|}  \right\}^{d^{\tau_n}} 
        \end{align*}
        Equation \eqref{rough_bound} implies that this expression is bounded above by $\exp \{ - e^{ \epsilon|a_n|}\}$ (since the fluctuations of the increments are decreasing). Therefore, by the assumption that $|a_n| \gg (\log n)^{\frac{1}{2}}$, these probabilities are summable. By the Borel-Cantelli lemma, it follows that almost surely:
        \begin{align}\label{eqn:exp_lower_bound}
        Y_n \geq e^{a_n - \epsilon|a_n|} \exp \big \{ \max_{|v|=\tau_n} \bar{V}_{\tau_n}^n (v)\big	\} \,,
        \end{align}
        for all sufficiently large $n$.

        However, it is well known that there exists an explicit constant $C > 0$ such that:\begin{equation*}
            \frac{1}{\tau_n } \max_{|v| =\tau_n }\{\bar{V}_{\tau_n}^n(v)\} \to C
        \end{equation*} almost surely (this maximum corresponds to the position of the rightmost particle in the time-homogeneous branching random walk system. In this case, since $\tau_n\ll n$, this result also holds for the first $\tau_n$ steps of the time-inhomogeneous environment). Thus,
        \[ \exp \big \{  \max_{|v| =\tau_n }\{\bar{V}_{\tau_n}^n(v)\} \big	\} = e^{C^*\epsilon |a_n| (1 + o(1))} \]
        for some constant $C^* > 0$ (independent of $\epsilon$). Hence, almost surely, $Y_n \geq e^{a_n(1+o(1))}$.
    \end{proof}

    To conclude the proof, we apply the following classical result in GMC theory, which is the so-called Kahane's concentration inequality (this inequality originated in \cite{Kahane}; see \cite{robert2008gaussianmultiplicativechaosrevisited} for a proof):
    
\begin{lemma} (Kahane's Concentration Inequality)\label{Kahaneineq}
    Let $F$ be a concave function such that:
    \begin{equation*}
        \forall x\in \mathbb{R}_+, |F(x)|\leq M\left(1+|x|^m\right)
    \end{equation*} 
    for some positive constants $M,m$, and let $\sigma$ be a Radon measure on the Borel subsets of $\mathbb{R}^d$. Given a bounded Borel set $A$, let $(X_r)_{r\in A}, (Y_r)_{r\in A}$ be two continuous centered Gaussian processes with continuous covariance kernels $K_X$ and $K_Y$ such that:
    \begin{equation*}
        \forall u,v\in A, K_X(u,v)\leq K_Y(u,v)
    \end{equation*}
    Then we have:
    \begin{equation*}
        \mathbb{E}\left[F\left(\int_A e^{X_r-\frac{1}{2}\mathbb{E}[X_r^2]}\sigma(dr)\right)\right]\geq \mathbb{E}\left[F\left(\int_A e^{Y_r-\frac{1}{2}\mathbb{E}[Y_r^2]}\sigma(dr)\right)\right]
    \end{equation*}
\end{lemma}
    
In this case, we may couple the law of the G-W tree for both time-homogeneous and time-inhomogeneous BRWs such that they have the same offspring distribution, and embed each realization of the G-W tree into the unit interval (i.e., $A=[0,1]$; see Section \ref{modelBRWs} for the detailed construction of such an embedding), and let the Borel measure $\sigma$ be the point mass associated to each realization. Note that for any $0<\alpha<1$, $F(x)=x^\alpha$ satisfies the conditions in the Lemma \ref{Kahaneineq}.
    
The upper bound of $\bar{W}_n^{\beta_c}$ follows by a similar argument as in Proposition \ref{Inhomogeneous upper bound}, except that Hölder's inequality is used here instead of the Cauchy-Schwarz inequality. We need to show that for any $\bar{\epsilon}>0$, we have $\mathbb{E}[(\bar{W}_n^\beta)^\alpha]\leq n^{-\frac{1}{2}+\bar{\epsilon}}$. For all $p,q>1$ satisfying $1/p+1/q=1$, and recalling the expectation $\mathbf{E}_{\lfloor n^{1-\epsilon} \rfloor}$ under the (random) measure $\bar{\mathbb{P}}^{\lfloor n^{1-\epsilon} \rfloor}$ defined in \ref{Random measure}, we have:
\begin{equation*}
    \mathbb{E}\left[(\bar{W}_n^\beta)^\alpha\right] \leq \mathbb{E}\left[(\bar{W}_n^{\lfloor n^{1-\epsilon} \rfloor})^{p\alpha}\right]^{1/p}\mathbb{E}\left[\left(\mathbf{E}_{\lfloor n^{1-\epsilon} \rfloor}\left[\tilde{\bar{W}}_n^{\lfloor n^{1-\epsilon} \rfloor}\right]\right)^{q\alpha}\right]^{1/q}
\end{equation*}

Note that for any $q>1$, the second term on the right-hand side of the above inequality equals 1. As a result, for any fixed $\alpha>0$ and $\epsilon>0$, we may let $p$ be arbitrarily close to 1 so that $p^2\alpha<1$. Since $f$ is continuous at 0, for any fixed $\epsilon>0$, when $n$ is large enough, there exists $\delta>0$ such that $\lfloor (1-\delta)n \rfloor<\sum_{i=1}^{\lfloor n^{1-\epsilon}\rfloor} \left(f(\frac{i}{n})-1\right)< n $. Then we can see that (where the inequality \ref{Kahaneineqused} follows from Kahane's concentration inequality \ref{Kahaneineq}):
\begin{equation*}
   \mathbb{E}\left[(\bar{W}_n^{n^{1-\epsilon}})^{p\alpha}\right]=\mathbb{E}\left[\left(\frac{1}{d^n}\sum_{\mathbf{|x|}=n} e^{\beta_c \bar{H}_n^{\lfloor n^{1-\epsilon}\rfloor}(\mathbf{x})-\frac{1}{2}\sum_{i=1}^{\lfloor n^{1-\epsilon}\rfloor-1} \beta_c^2}\right)^{p\alpha}\right]
\end{equation*}
\begin{equation}\label{Kahaneineqused}
    \leq\mathbb{E}\left[\left(\frac{1}{d^n}\sum_{|\mathbf{x}|=n} e^{(1-\delta)\beta_cH_n(\mathbf{x})-\frac{1}{2}\lfloor (1-\delta)n \rfloor \beta_c^2}\right)^{p\alpha}\right]
\end{equation}
\begin{equation*}
    \leq \mathbb{E}\left[\left(\frac{1}{d^{\lfloor (1-\delta)n \rfloor}}\sum_{|\mathbf{x}|=\lfloor (1-\delta)n \rfloor} e^{\beta_cH_{\lfloor (1-\delta)n \rfloor}(\mathbf{x})-\frac{1}{2}\lfloor (1-\delta)n \rfloor \beta_c^2}\right)^{p^2\alpha}\right]^{1/p}\mathbb{E}\left[\left(\mathbf{E}_{\lfloor (1-\delta)n \rfloor}[\tilde{W}_n^{\lfloor (1-\delta)n \rfloor}]\right)^{pq\alpha}\right]^{1/q}
\end{equation*}
\begin{equation*}
    =\mathbb{E}\left[\left(\bar{W}_{\lfloor (1-\delta)n \rfloor}^\beta\right)^{p^2\alpha}\right]^{1/p}\times 1
\end{equation*}
where the inequality Kahane's concentration inequality (Lemma \ref{Kahaneineq}),in the expressions above, analogous to that in the proof of Theorem \ref{Inhomogeneous upper bound}, we define $\mathbf{E}_{\lfloor (1-\delta)n \rfloor}$ to be the expectation under the transformed measure $\bar{\mathbb{P}}^{\lfloor (1-\delta)n \rfloor}$ (recall $\mathbf{P}^*$ is the survival probability of GW tree):
\begin{equation}
    \frac{\mathrm{d}\bar{\mathbb{P}}^{\lfloor (1-\delta)n \rfloor}(\mathbf{x})}{\mathrm{d}\mathbf{P}^*} = \frac{e^{-\frac{1}{2}\beta^2 \lfloor (1-\delta)n \rfloor + \beta H_n^{\lfloor (1-\delta)n \rfloor}(\mathbf{x})}}{W_n^{\lfloor (1-\delta)n \rfloor}}
\end{equation}
The ``first step'' partition function denoted $W_n^{\lfloor (1-\delta)n \rfloor}$ is defined as:
\begin{equation*}
    \tilde{W}_n^{\lfloor (1-\delta)n \rfloor}=\frac{1}{d^{\lfloor (1-\delta)n \rfloor}}\sum_{|\mathbf{x}|=n} e^{\beta_cH_n^{\lfloor (1-\delta)n \rfloor}(\mathbf{x})-\frac{1}{2}\lfloor (1-\delta)n \rfloor\beta_c^2} \text{ where } H_n^{\lfloor (1-\delta)n \rfloor}(\mathbf{x})=\sum_{i=1}^{\lfloor (1-\delta)n \rfloor} w(x_i)
\end{equation*} 
and its ``complement'' denoted $\tilde{\bar{W}}_n^{\lfloor (1-\delta)n \rfloor}$ is defined as:
\begin{equation*}
    \tilde{\bar{W}}_n^{\lfloor (1-\delta)n \rfloor}=\frac{1}{d^{\lceil \delta n \rceil}}\sum_{\mathbf{y}\in D_{\lfloor (1-\delta)n \rfloor}^\mathbf{x}} e^{\beta_cH_n^{\lfloor (1-\delta)n \rfloor\to n}(\mathbf{y})-\frac{1}{2}\lceil \delta n \rceil\beta_c^2}, \text{ where } H_n^{\lfloor (1-\delta)n \rfloor\to n}(\mathbf{x})=\sum_{i=\lfloor (1-\delta)n \rfloor}^n w(x_i)
\end{equation*}
    
As a result, since the asymptotic behavior of $W_n^{\beta_c}$ is $n^{-\frac{1}{2}+o(1)}$, we have (by the dominated convergence theorem) the convergence of the fractional moment:
\begin{equation*}
    \mathbb{E}\left[(\bar{W}_n^{n^{1-\epsilon}})^{p^2\alpha}\right]^{1/p}\leq\mathbb{E}\left[\left(\bar{W}_{\lfloor (1-\delta)n \rfloor}^\beta\right)^{p^2\alpha}\right]^{1/p^2}\leq \left(\lfloor (1-\delta)n \rfloor\right)^{-\frac{1}{2p^2}+o(1)}\leq n^{-\frac{1}{2p^2}(1-\delta)+o(1)}
\end{equation*}
    
Then we can derive the upper bound for the fractional moment:
\begin{equation*}
    \limsup_{n\to\infty}\mathbb{E}\left[(\bar{W}_n^\beta)^\alpha\right]\leq n^{-\frac{1}{2p^2}(1-\delta)+o(1)}
\end{equation*}
which holds for any $p>1$ and $\epsilon>0$. Letting $p\to 1$ and using the fact that $\delta$ can be made arbitrarily small by taking $\epsilon$ sufficiently small yields the desired upper bound.
    
Similarly, the lower bound is obtained via a direct comparison of the fractional moments of the two partition functions:
\begin{equation*}
    \mathbb{E}\left[\left(\bar{W}_n^\beta\right)^\alpha\right]\geq \mathbb{E}\left[\left(W_n^\beta\right)^\alpha\right]
\end{equation*}

This finishes the proof of the asymptotic behavior of $\bar{W}_n^\beta$ at $\beta=\bar{\beta}_c$.
\end{proof}

\section{Further Directions}\label{Future}

Here, we briefly contextualize our established results within the landscape of related current research areas (including BRWs, directed polymers and GMC) and outline potential directions for future work:

\subsection{Behavior of Partition Function at Criticality}

The derivative martingale is another key martingale associated with branching random walks and plays a central role in the study of their extreme values (see Section 5.4 of \cite{Shi2015} for more details). At the critical parameter $\beta=\beta_c$, Biggins and Kyprianou \cite{Biggins_Kyprianou_2004} showed that the derivative martingale (as the name suggests, by taking ``derivative'' with respect to $\beta$ of $W_n^\beta$ at $\beta_c$), defined by:  
\begin{equation*}  
    D_n:=\sum_{|\mathbf{x}|=n} \left(\beta_c H_n(\mathbf{x})-\frac{\beta_c^2 n}{2}\right)e^{\beta_c H_n(\mathbf{x})-\frac{\beta_c^2 n}{2}}  
\end{equation*}  
converges almost surely to a non-degenerate limit $D_\infty>0$. Furthermore, Elie and Shi \cite{ElieAidekonZhanShiAOP809} established that $D_\infty$ has the same law as a constant multiple of the limit in probability of $n^{\frac{1}{2}}W_n^{\beta_c}$ (the constant depends only on the law of the environment and the Galton-Watson tree):  

\begin{theorem}[\cite{ElieAidekonZhanShiAOP809}, Theorem 1.1]
We have the Seneta–Heyde scaling for Derivative martingale $D_\infty$:  
    \begin{equation*}  
        \lim_{n\to\infty} n^{\frac{1}{2}}W_n^{\beta_c}= \left(\frac{2}{\pi\sigma^2}\right)^{\frac{1}{2}}D_\infty  
    \end{equation*}  
where $\sigma^2$ depends on the law of environment and G-W tree only. 
\end{theorem}  

While one may conjecture that a similar scaling result holds in time-inhomogeneous weights, it is worth mentioning the difference between the model in question and the time-homogeneous case. Either in the literature of critical GMC or in BRWs, the behavior of the partition function at criticality is governed by the extreme value (of the Gaussian field/BRWs). As we mentioned before, in the sub-critical regime ($\gamma<\sqrt{2d}$), the GMC measure is almost surely supported on $\gamma$-thick points. At the critical parameter $\gamma=\sqrt{2d}$, it corresponds to the maximal possible thickness:\begin{equation*}
    \lim_{\epsilon\to 0 }\sup_{x\in D} X_\epsilon(x)/|\log\epsilon|\to \sqrt{2d} \text{  in probability}
\end{equation*}

In the literature of BRWs, it was shown in \cite{Ming-Zeitouni} that when the variance is decreasing, the extreme value of BRWs with decreasing variance exhibits different asymptotic behavior from their constant variance counterparts. However, the asymptotics of the extreme value of the time-inhomogeneous BRWs studied in the present paper remains an open problem, which would be interesting to explore in the future.

\subsection{Relationship with Two-Dimensional Directed Polymers}
Recall that we discussed the heuristic relationship between two-dimensional directed polymers and BRWs with decreasing variance in Section \ref{modelBRWs}; it would be interesting to establish a rigorous connection between the time-inhomogeneous BRWs with decreasing variance studied in this paper and the log-partition function of two-dimensional directed polymers. In fact, this heuristic has been confirmed in \cite{cosco2025maximumdimensionalgaussiandirected} for the extreme values of $h_N(x)$ (the partition function under certain scaling and normalization), and raises an intriguing open problem: whether the two-dimensional polymer measure defined by
\begin{equation*}
\mu_N^\beta:=\frac{e^{\beta h_N}}{\mathbb{E}\left[e^{\beta h_N}\right]}
\end{equation*}
converges to a GMC measure. The universality of partition functions in the present paper (Theorem \ref{Main2}) suggests an affirmative answer to this problem (posed in \cite{cosco2025maximumdimensionalgaussiandirected}), though additional technical work is required to prove this rigorously.

\subsection{GMC Beyond Log-correlation}

 An interesting implication of Theorem \ref{Main2} is that if we embed the (binary) BRW into the unit interval $[0,1]$, the covariance structure $K(x,y)$ (for $x,y\in [0,1]$) differs from that of GMC (which takes the form $-\log |x-y|+O(1)$): it is not log-correlated. The universality behavior we established enables us to investigate whether we can construct a GMC-type object beyond log-correlated fields, which defines a variant of GMC measure. 

\subsection{Other Related Models}

A particularly natural extension is the continuous time analogy of BRWs: the time-inhomogeneous Branching Brownian Motions (BBM). BBM were studied earlier than BRWs and are of independent interest (see \cite{Millard16} for a precise introduction for its time-inhomogeneous counterpart). Indeed, many results in branching random walks are known in the case of the branching Brownian motion,  and the modifications of methods are more or less painless (and sometimes even neater). Conversely, once a result is established in Branching Brownian motions, it was often believed that a corresponding result holds in the context of BRWs (though not always straightforward to extend). 

On the other hand, the model of BRWs has recently been viewed as a specific example of directed polymers on infinite graphs (see \cite{Cosco_2021} for an introduction to this topic), and we hope our results can be extended to other related models, such as $\lambda$-biased BRWs on Galton-Watson trees. Finally, it is also interesting to extend the results to environments that are not independent: i.e., time-correlated random fields. There are more interesting dependence of time, compared to macroscopic decay in variance discussed in the present paper in the modern literature. For example, the study of the influence of environmental fluctuations, or stochastic impurities, on directed polymers has been initiated in \cite{chen2024localizationregimehighdimensionaldirected}. The authors sincerely hope that more interesting properties and phenomena will be discovered in the time-inhomogeneous settings of these models.

\printbibliography[
  title={References},  
  heading=bibnumbered  
]
\end{document}